\pdfoutput=1
\documentclass[reqno, noend, a4paper,oneside]{amsart}

\usepackage{amsthm,amsmath,amsfonts,mathrsfs,amssymb,todonotes,verbatim}
\usepackage{algorithm}
\usepackage{algorithmicx, algpseudocode}
\usepackage[T1]{fontenc}   
\usepackage{multirow}
\usepackage{url}
\usepackage{multirow}
\usepackage{rotating}
\usepackage{enumerate}
\usepackage{natbib}
\usepackage{epstopdf}
\usepackage{graphicx}
\usepackage{enumitem}
\usepackage{arydshln}
\usepackage[hang]{caption}
\usepackage[hyperfootnotes=false]{hyperref}
\usepackage{setspace}
\usepackage[textheight=640pt]{geometry}
\usepackage{calc}

\algdef{SE}[DOWHILE]{Do}{doWhile}{\algorithmicdo}[1]{\algorithmicwhile\ #1}%

\theoremstyle{plain}
\newtheorem{theorem}{Theorem}[section]
\newtheorem{corollary}[theorem]{Corollary}
\newtheorem{defn}[theorem]{Definition}
\newtheorem{prop}[theorem]{Proposition}
\newtheorem{lemma}[theorem]{Lemma}
\theoremstyle{definition}
\newtheorem{remark}[theorem]{Remark}

\newtheorem{notation}[theorem]{Notation}
\newtheorem{example}[theorem]{Example}

\newcommand{\Singular}{\textsc{Singular}}

\usepackage{accents}

\DeclareMathOperator{\reg}{reg}

\DeclareMathOperator{\Mon}{mon}

\DeclareMathOperator{\ord}{ord}
\DeclareMathOperator{\id}{id}
\DeclareMathOperator{\parm}{par}
\DeclareMathOperator{\m}{\mathfrak{m}}
\DeclareMathOperator{\jet}{jet}
\DeclareMathOperator{\corank}{corank}
\DeclareMathOperator{\supp}{supp}
\DeclareMathOperator{\Jac}{Jac}

\DeclareMathOperator{\dash}{-}

\DeclareMathOperator{\NF}{NF}
\DeclareMathOperator{\F}{F}

\DeclareMathOperator{\N}{\mathbb{N}}
\DeclareMathOperator{\Q}{\mathbb{Q}}
\DeclareMathOperator{\R}{\mathbb{R}}
\DeclareMathOperator{\C}{\mathbb{C}}

\DeclareMathOperator{\coeff}{coeff}

\DeclareMathOperator{\sat}{sat}


\title[Classification of Singularities with Non-Degenerate Newton Boundary]%
{Classification of Complex Singularities with Non-Degenerate Newton Boundary}

\author{Janko B\"ohm}
\address{Janko B\"ohm\\
Department of Mathematics\\
University of Kaiserslautern\\
Erwin-Schr\"odinger-Str.\\
67663 Kaiserslautern\\
Germany}
\email{boehm@mathematik.uni-kl.de}

\author{Magdaleen S.\@ Marais}
\address{Magdaleen S.\@ Marais\\
University of Pretoria and African Institute for Mathematical Sciences\\
Department of Mathematics and Applied Mathematics\\
Private bag X20\\
Hatfield 0028\\
South Africa}
\email{magdaleen.marais@up.ac.za}

\author{Gerhard Pfister}
\address{Gerhard Pfister\\
Department of Mathematics\\
University of Kaiserslautern\\
Erwin-Schr\"odin\-ger-Str.\\
67663 Kaiserslautern\\
Germany}
\email{pfister@mathematik.uni-kl.de}

\thanks{This research was supported by grant KIC14081491583 of the National Research Foundation (NRF) of South Africa, the Staff Exchange Bursary Programme  of the University of Pretoria, a Research Development Grant awarded by the University of Pretoria, by the AIMS institute, and by the German Research Foundation (DFG) through Project II.5 of SFB-TRR 195.}


\begin{document}

\begin{abstract}
In his groundbreaking work on classification of singularities with regard to right and stable equivalence of germs, Arnold has listed normal forms for all isolated hypersurface singularities over the complex numbers with either modality $\leq 2$ or Milnor number $\leq 16$. Moreover, he has described an algorithmic classifier, which determines the type of a given such singularity. 
In the present paper, we extend Arnold's work to a large class of singularities which is unbounded with regard to modality and Milnor number. We develop an algorithmic classifier, which determines a normal form for any corank $\leq 2$ singularity which is equivalent to a germ with non-degenerate Newton boundary in the sense of Kouchnirenko. 
In order to realize the classifier, we prove a normal form theorem: Suppose $K$ is a $\mu$-constant stratum of the jet space which contains a germ with a non-degenerate Newton boundary. 
We first observe that all germs in $K$ are equivalent to some germ with the same fixed non-degenerate Newton boundary. We then prove that all right-equivalence classes of germs in $K$ can be covered by a single normal form obtained from a regular basis of an  appropriately chosen special fiber.  All algorithms are
implemented in the library \texttt{arnold.lib} for the computer algebra system \textsc{Singular}. 

\end{abstract}

\maketitle

\section{Introduction}\label{Introduction}

In \citep{A1976, AVG1985} Arnold has classified all isolated hypersurface singularities for which 
the Milnor number $\le 16$ or for which the modality is $\le 2$. He has given normal forms in the sense of polynomial families with moduli parameters such that every stable equivalence class of function germs contains at least one, but only finitely many, elements of these families. We refer to such elements as normal form equations. In Arnold's work and in the rest of the paper each normal form contains representatives of all the stable equivalence classes of a specific $\mu$-constant stratum. For a specific germ $f$, we refer to a normal form covering the $\mu$-constant stratum containing $f$, as a normal form of $f$. Two germs are stably equivalent if they are right-equivalent after the direct addition of a non-degenerate quadratic form (see Section 5, \citet{A1976}). Two function germs $f,g\in\m^2\subset\C\{x_1,\ldots,x_n\}$, where $\m=\langle x_1,\ldots,x_n\rangle$, are right-equivalent, written $f\sim g$, if there is a $\C$-algebra automorphism $\phi$ of $\C\{x_1,\ldots,x_n\}$ such that $\phi(f)=g$. Using the Splitting Lemma (Lemma 9.2.10, \citet{PdJ2000}), any germ with an isolated singularity at the origin can be written, after choosing a suitable coordinate system, as the sum of two functions on disjoint sets of variables. One function that is called the non-degenerate part, is a non-degenerate quadratic form, and the other part, called the residual part, is in $\m^3$. The Splitting Lemma is implemented in the computer algebra system \Singular\ \citep{DGPS, GP2008} as part of the library \texttt{classify.lib} \citep{classify}. Given the classification of singularities of function germs, one can easily derive the classification of the corresponding hypersurface singularities (up to contact equivalence, which in addition to right equivalence allows for multiplication of the defining equation by a constant), see, for example, \citet{realclassify1}.

In \citep{AVG1985}, Arnold has made his classification explicit by describing an algorithmic classifier, which is based on a case-by-case analysis formulated in a series of $105$ theorems. This approach determines the type of the singularity in the sense of its normal form. The algorithm does not determine the values of the moduli parameters, that is, no normal form equation is given. 
Arnold's classifier is implemented in the \Singular\ library \texttt{classify.lib}. In \citep{BMP2016} a classifier has been developed for complex singularities of modality $\leq 2$ and corank $\leq 2$, which computes, for a given input polynomial, a normal form equation in its equivalence class. For an implementation see the \Singular\ library \texttt{classify2.lib} \cite{classify2}.

Classification of complex singularities has a multitude of practical and theoretical applications. In particular, the classification of real singularities as for example developed in \citep{realclassify1, realclassify2,realclassify3} is based on determining first the complex type of the singularity. Recent applications arise in high-energy physics in the study of singularities of Baikov polynomials \citep{LP2013}, which are a key object in the computation of integration-by-parts identities for Feynman integrals \citep{ibp}.

In the present paper, we develop a classification algorithm for all isolated hypersurface singularities of corank $\leq 2$ which are right-equivalent to a germ with a non-degenerate Newton boundary in the sense of Kouchnirenko's definition in  \cite{K1976}. This algorithm determines for a given input polynomial $f$ representing a $\mu$-constant stratum a polynomial normal form (that is, a family of polynomials with polynomial, in fact linear, dependence on the moduli parameters) which covers the whole $\mu$-constant stratum of $f$.  

In order to realize the classifier, we prove a normal form theorem: Suppose $K$ is a $\mu$-constant stratum of the jet space which contains a germ with a non-degenerate Newton boundary.  
We observe that all germs in $K$ are equivalent to some germ with the same fixed non-degenerate Newton boundary. We then prove that all right-equivalence classes of germs in $K$ can be covered by a single normal form obtained from  a so-called regular basis of the Milnor algebra of an  appropriately chosen special fiber
.\footnote{We remark, that we have also developed an algorithm for determining a normal form equation for $f$, that is, a representative of its right-equivalence class by specifying the values of the moduli parameters in our normal form family. We also have developed an algorithm to enumerate all normal form families up to a specified Milnor number. These algorithms will be presented in follow-up papers.}
All our algorithms are implemented in the \textsc{Singular} library \texttt{arnold.lib} \cite{arnoldlib}.

This paper is structured as follows: In Section \ref{Definitions and Preliminary Results}, we give the fundamental definitions and provide the prerequisites on singularities and their classification. 
In Section~\ref{section:NormalForm}, we prove that a normal form of the $\mu$-constant stratum of a given germ which is equivalent to a germ with non-degenerate Newton boundary can be obtained from the Newton boundary and a regular basis of the Milnor algebra (see Theorem~\ref{theorem:normalform}). 
In Section~\ref{sec char exp} we prove a self-contained fact which is required to complete the proof of the normal form theorem. We show  that, if two germs have the same characteristic exponents and intersection numbers, and one of them has a non-degenerate Newton boundary, then also the other one is right-equivalent to a germ with non-degenerate Newton boundary and same Newton polygon. Moreover, we prove that if a germ $f$ is equivalent to a germ with a non-degenerate Newton boundary, the Newton polygon of this germ is uniquely determined by the characteristic exponents and intersection numbers of the branches of $f$, provided we require the normalization condition.
In Section~\ref{section:generalAlg}, we develop an algorithm which transforms an input singularity to a germ with a non-degenerate Newton boundary, in case this is possible (see Algorithm~\ref{alg:clas}). 
In Section \ref{section:modality}, we use the result of Algorithm \ref{alg:clas} to compute the modality for the input polynomial, as well as a regular basis of the Milnor algebra (see Algorithms~\ref{alg:mod} and \ref{alg:regularbasis}). 
Finally, in Section~\ref{section:normalform}, relying on Theorem~\ref{theorem:normalform}, we use Algorithm~\ref{alg:clas} and Algorithm~\ref{alg:regularbasis} to compute for a given input polynomial $f$ a normal form which covers its $\mu$-constant stratum (see Algorithm~\ref{alg:normalform}).\medskip

\noindent\emph{Acknowledgements.} We would like to thank Gert-Martin Greuel 
and Hans Sch\"onemann for helpful discussions. We would also like to thank Fabian M\"aurer for contributing to an implementation under development based on the computer algebra system \textsc{OSCAR}.

\section{Definitions and Preliminary Results}\label{Definitions and Preliminary Results}
In this section we give some basic definitions and results, as well as some notation that will be used throughout the paper. We denote by $\mathbb{C}\{x_1,\ldots,x_n\}$ the ring of power series convergent in open neighborhoods of $(0,\ldots,0)$.
We start with a short account on weighted jets, filtrations, and Newton polygons. See also
\citep{A1974} and \citep{PdJ2000}. 

\begin{notation}
We write $\Mon(x_1,\ldots,x_n)$ for the multiplicative monoid of monomials in the variables $x_1,\ldots,x_n$. Given  $f\in\mathbb{C}\{x_1,\ldots,x_n\}$ and $m\in \Mon(x_1,\ldots,x_n)$, we write $\coeff(f,m)$ for the coefficient of $m$ in $f$. 
\end{notation}

\begin{defn}
Let $w=(c_1,\ldots,c_n)\in\N^n$ be a weight on the variables $(x_1,\ldots,x_n)$. The $w$-weighted degree on $\Mon(x_1,\ldots,x_n)$ is given by $w\dash\deg(\prod_{i=1}^nx_i^{s_i}):=\sum_{i=1}^n c_i s_i$. If the weight of all variables is equal to $1$, we refer to the weighted degree of a monomial $m$ as the standard degree of $m$ and write $\deg(m)$ for $w\dash\deg(m)$. We use the same notation for terms of polynomials.

We call a polynomial $f\in \C[x_1,\ldots, x_n]$ \textbf{quasihomogeneous} or \textbf{weighted homogeneous} of degree $d$ with respect to the weight $w$
if $w\dash\deg(t)=d$ for any term $t$ of $f$.
\end{defn}

\begin{defn}\label{def:piecewiseWeight}
Let $w=(w_1,\ldots,w_s)\in(\N^n)^s$ be a finite family of weights on the variables $(x_1,\ldots,x_n)$. For any monomial (or term) $m\in \C[x_1,\ldots,x_n]$,  we define the \textbf{piecewise weight} with respect to $w$ as
\begin{eqnarray*}
w\dash\deg(m)&:=&\min_{i=1,\ldots,s}w_i\dash\deg(m).\\
\end{eqnarray*}
A polynomial $f$ is called  \textbf{piecewise homogeneous} of degree $d$ with respect to $w$ if $w\dash\deg(t)=d$ for any term $t$ of $f$. 
\end{defn}
\pagebreak[3]

\begin{defn}
Let $w$ be a (piecewise) weight on $\Mon(x_1,\ldots,x_n)$.
\begin{enumerate}[leftmargin=10mm]
\item
Let $f = \sum_{i = 0}^{\infty} f_{i}$ be the decomposition of
$f \in \C\{x_1,\ldots,x_n\}$ into weighted homogeneous summands $f_{i}$ of
$w$-degree $i$. If $f_i=0$ for all $i\geq d$ and $f_d\neq 0$, then we write $\deg_{w}(f)=d$. The \textbf{(piecewise) weighted $j$-jet} of $f$ with respect to $w$ is
\[
w \dash \jet(f, j) := \sum_{i = 0}^j f_{i} \,.
\]
The sum of terms of  $f$ of lowest $w$-degree is the \textbf{principal part} of $f$ with respect to $w$.

\item  We say that a power series in $\C\{x_1,\ldots,x_n\}$ has \textbf{filtration} $d \in \N$ with respect to $w$ if all its monomials are of $w$-weighted degree $d$ or higher. We write $E_d^w$  for the sub-vector space of power series  in $\C\{x_1,\ldots,x_n\}$ of filtration $d$ with respect to $w$. The sub-vector spaces $E_d^w$ for $d \in \N$ define a filtration on $ \C\{x_1,\ldots,x_n\}$. 
\end{enumerate}
\end{defn}
If the weight
of each variable is $1$, we write $E_d$ instead of  $E_d^w$ and  $\jet(f,j)$ instead of $w\dash\jet(f,j)$. 
\begin{defn}
A power series $f\in\m^2\subset \C\{x_1,\ldots,x_n\}$ is \textbf{ $k$-determined} if
\[ f\sim \jet(f,k)+g\qquad\text{for all } g\in E_{k+1},\]
where $\sim$ denotes right-equivalence. We define the \textbf{determinacy} $\operatorname{dt}(f)$ of $f$ as the minimum number $k$ such that $f$ is $k$-determined.

\end{defn}

    \begin{defn}
Let $f\in\mathbb C\{x_1,\ldots,x_n\}$. The {\bf Jacobian ideal} $\Jac(f)\subset \C\{x_1,\ldots,x_n\}$ of $f$ is generated by the partial derivatives of $f$. The {\bf local algebra} $Q_f$ of $f$ is the residue class ring of the Jacobian ideal of $f$. The \textbf{Milnor number} of $f$ is the $\mathbb C$-vector space dimension of $Q_f$.
 \end{defn}
 
 It is a well-known fact, that all function germs are finitely determined:

\begin{theorem}\cite[Theorem 9.1.4]{PdJ2000}\label{thm fin det}
Let $f\in\m^2\subset \C\{x_1,\ldots,x_n\}$. If $\m^{k+1}\subset \m^2 \cdot \Jac(f)$ then $f$ is \textbf{ $k$-determined}. In particular, isolated singularities are finitely determined, and if $k\geq \mu(f)+1$ then $f$ is $k$-determined.
\end{theorem}

\begin{remark}\label{rem Artin}Theorem \ref{thm fin det} implies that every right-equivalence class of an isolated singularity has a polynomial representative. Note that since a result analogous to Theorem \ref{thm fin det} is true also for formal power series, every right-equivalence class of formal power series with an isolated singularity has a polynomial (hence convergent) representative.  

There is  a bijection between the right-equivalence classes in the convergent and the formal sense: By the Artin approximation theorem (\cite{artin}), two convergent power series which are formally equivalent are also equivalent via a convergent transformation.

\end{remark}

\begin{defn}
Let $w\in \mathbb{N}^n$ be a single weight.  A power series $f\in \C\{x_1,\ldots,x_n\}$ is called \textbf{semi-quasihomogeneous} with respect to $w$ if its principal part with respect to $w$ is \textbf{non-degenerate}, that is, has finite Milnor number.\footnote{We say that $f$ is \textbf{(semi-)quasihomogeneous} if there exists a weight $w$ such that $f$ is (semi-)quasihomogeneous with respect to $w$.} The principal part of $f$ is then called the \textbf{quasihomogeneous part} of $f$.

\end{defn}

There are similar concepts of jets and filtrations for coordinate transformations:

\begin{defn}\label{phi}
Let $\phi$ be a $\C$-algebra automorphism of $\C\{x_1,\ldots,x_n\}$ and let
$w$ be a single weight on $\Mon(x_1,\ldots,x_n)$.

\begin{enumerate}[leftmargin=10mm]
\item
For $j > 0$ we define \emph{$w\dash\jet(\phi,j):=\phi_j^w$} as the automorphism given by
\[
\phi_j^w(x_i) := w\dash\jet(\phi(x_i),w\dash\deg(x_i)+j) \quad
\text{for all }i = 1,\ldots,n \,.
\]
If the weight of each variable is equal to $1$, that is, $w = (1, \ldots, 1)$, we
write $\phi_j$ for $\phi_j^w$.

\item\label{enum:filtration}
$\phi$ has filtration $d$ if, for all $\lambda \in \N$,
\[
(\phi-\id)E_\lambda^w \subset E_{\lambda+d}^w \,.
\]
\end{enumerate}
\end{defn}

\begin{remark}
Note that $\phi_0(x_i) = \jet(\phi(x_i), 1)$ for all $i = 1, \ldots, n$.
Furthermore note that $\phi_0^w$ has filtration $\le 0$, and
that, for $j > 0$, $\phi_j^w$ has filtration $j$ if $\phi_{j-1}^w = \id$.
\end{remark}

 \begin{defn}\label{def NB}
 Let $f=\sum_{i_1,\ldots,i_n}a_{i_1,\ldots,i_n}x_1^{i_1}\cdots x_n^{i_n}\in\C\{x_1,\ldots,x_n\}$. We call
 \begin{eqnarray*}
 \Mon(f)&:=&\{x_1^{i_1}\cdots x_n^{i_n}\ |\ a_{i_1,\ldots,i_n}\neq 0\}\\
 \end{eqnarray*}
 and
 \begin{eqnarray*}
 \sup(f)&:=&\{{i_1}\cdots {i_n}\ |\ a_{i_1,\ldots,i_n}\neq 0\}\\
 \end{eqnarray*}
 the \textbf{monomials} of $f$ and the \textbf{support} of $f$, respectively. Let
\begin{eqnarray*}
\Gamma_+(f)&:=&\displaystyle{\bigcup_{x_1^{i_1}\cdots x_n^{i_n}\in\supp(f)}}((i_1,\ldots,i_n)+\R^n_+)\\
\end{eqnarray*}
and let $\Gamma(f)$ be the boundary in $\R^n_+$ of the convex hull of $\Gamma_+(f)$.  We call $\Gamma(f)$ the {\bf Newton boundary} of $f$.
Then:
\begin{enumerate}[leftmargin=10mm]
\item Compact segments of $\Gamma(f)$ are called {\bf facets} of $\Gamma(f)$\footnote{This term is used since they are codimension $1$ faces of $\Gamma_+(f)$}. If $\Delta$ is a facet, then the set of monomials of $f$ lying on $\Delta$ is denoted by $\supp(f,\Delta)$ and the sum of the terms lying on $\Delta$ by $\jet(f,\Delta)$. Moreover, we write $\supp(\Delta)$ for the set of monomials corresponding to the lattice points of $\Delta$.  We use the same notation for a set of facets, considering the monomials lying on the union of the facets.
\item Any facet $\Delta$ induces a weight $w(\Delta)$ on $\Mon(x_1,\ldots,x_n)$ in the following way: If $\Delta$ has the normal vector $-(w_{x_1},\ldots, w_{x_n})$, in lowest terms, and $w_{x_1},\ldots,w_{x_n}>0$, we set \[w(\Delta)\dash\deg(x_1)= w_{x_1},\ldots, w(\Delta)\dash\deg(x_n)=w_{x_n}.\] 
 \item If $w_1,\ldots,w_s$ are the weights associated to the facets of $\Gamma (f)$ ordered by increasing slope, there are unique minimal integers $\lambda_1,\ldots,\lambda_s\geq 1$ such that the piecewise weight with respect to $w(f):=(\lambda_1 w_1,\ldots,\lambda_s w_s)$ 
 is constant on $\Gamma(f)$.  We denote this constant by~$d(f)$.
\item Let $\Delta_1$ and $\Delta_2$ be facets with weights $w_1$ and $w_2$, respectively, and let $w$ be the piecewise weight defined by $w_1$ and $w_2$. Let $d$ be the $w$-degree of the monomials on $\Delta_1$ and $\Delta_2$. Then $\operatorname{span}(\Delta_1,\Delta_2)$ is the Newton polygon associated to the sum of all monomials of $(w_1,w_2)$-degree $d$.
\item Suppose $\Gamma(f)$ has at least one facet. A monomial $m$ lies strictly underneath, on or above $\Gamma(f)$, if the $w(f)$-degree of $m$ is less than, equal to or greater than $d(f)$, respectively. 
 \end{enumerate}

   \end{defn}

 \begin{defn} \label{defn:regularBasis}
Suppose $f$ has finite Milnor number. A homogeneous basis $\{e_1,\ldots,e_{\mu}\}$ of the local algebra of $f$  is {\bf regular} with respect to the filtration given the piecewise weight $w$, if, for each $D\in\mathbb N$, the elements of the basis of degree $D$ with respect to $w$ are independent modulo the sum of $\Jac(f)$ and the space $E^w_{>D}$ of functions of filtrations bigger than $D$.

For convenience we use the following notation: A \textbf{regular basis for $f$} is a basis of $Q_f$ which is regular with respect to the filtration given by $w(f)$.
 \end{defn}
 \begin{remark}Suppose $f$ has finite Milnor number.
\begin{enumerate}[leftmargin=10mm]
\item  For any $f\in \C\{x_1,\ldots,x_n\}$ there exists a finite  basis of $Q_f$ with monomial representatives (take the monomials not appearing in the lead ideal of a Gr\"obner basis of $\Jac(f)$). We refer to this basis as a \textbf{monomial basis} of $Q_f$.
 
\item  A  monomial basis of $Q_f$ is regular for $f$ if and only if the images of the degree $D$ elements in the basis  form a basis for $E^w_D/(E^w_D\cap\Jac(f)+E^w_{>D})$  for each~$D$, where $w=w(f)$.

\end{enumerate}
 \end{remark}
 
 \begin{remark}\label{remark:regularBasis}
In the case that $f$ is semi-quasihomogeneous, any basis of $Q_f$ is a regular basis for $f$.  
 \end{remark}
 
 More generally, the following result is proven as Proposition 9.4 in \citet{A1974}.
 \begin{prop}
 For each $f\in\mathbb C\{x_1,\ldots,x_n\}$ with finite Milnor number, there exists a finite regular basis for $f$, in fact, one consisting entirely out of monomials.
 \end{prop}
 
 The following remark describes how to obtain a regular basis:
 
 \begin{remark}
 Under $ E^w_D/(E^w_D\cap\Jac(f)) \rightarrow E^w_D/((E^w_D\cap\Jac(f))+E^w_{>D})$, every element of a monomial basis of the target has a unique preimage with a monomial representative of degree $D$. Let $R_D$ be a distinct set of monomial representatives of degree~$D$ of these preimages. Then $\bigcup_D R_D$ is a regular basis for $f$. We may restrict the union to the $D$ with $E^w_D\not\subset\Jac(f)$.
 \end{remark}
  \begin{defn}
A germ $f\in\mathbb C\{x_1,\ldots,x_n\}$ is called \textbf{convenient} if the Newton polygon meets all coordinate axes.
 \end{defn}

    \begin{defn}
We say that a convenient germ $f\in\mathbb C\{x_1,\ldots, x_n\}$ has \textbf{non-degenerate Newton boundary} if for every facet $\Delta$ of $\Gamma(f)$ the saturation\footnote{The saturation of a polynomial $g\in\mathbb C[x,y]$ is the polynomial $\sat(g):=h\in\mathbb C[x,y]$ with $g=h\cdot x^ny^m$ with $n,m$ maximal. Note that $\sat(g)$ is a generator of the ideal $\langle g\rangle:\langle x,y\rangle^\infty:=\{h\in\mathbb C[x,y]\mid h\cdot\langle x,y\rangle^n\subset\langle g\rangle \text{ for some } n\ge 1\}$ in $\mathbb C[x,y]$.} of $\jet(f,\Delta)$ has finite Milnor number. 
 \end{defn}

  \begin{theorem}\label{theorem:milnorNumber} (\cite{K1976})
Let $f\in \mathbb C\{x,y\}$ be a germ and suppose that its Newton polygon meets the $x$- and $y$-axis at $(a,0)$ and $(0,b)$, respectively. Denoting by $S$ the area of the polygon formed by the $x$- and $y$-axis of the Newton boundary, we have $\mu(f)\geq 2\cdot S-a-b+1$. If $f$ has a non-degenerate Newton boundary, then $\mu(f)=2\cdot S-a-b+1$.
We define the \textbf{Newton number} of $f$ as $N(f)=2\cdot S-a-b+1$.
\end{theorem}
 
 \begin{defn}\label{def muconstant}
 The $\mu$\textbf{-constant stratum} of a germ $f\in\mathbb C\{x_1,\ldots,x_n\}$ with Milnor number $\mu = \mu(f)$ is the connected component of the $(\mu+1)$-jet space (that is, the $\C$-vector space of polynomials in $\C [x_1,\ldots,x_n]$ of degree $\leq \mu +1$) with fixed Milnor number $\mu$ which  contains $\jet(f,\mu +1)$.

  \end{defn}

Our work will be based on Arnold's notion of a normal form (see \cite{AVG1985}):

 \begin{defn}\label{def nfequ}
Let $K\subset \C\{x_1,\ldots, x_n\}$ be a union of equivalence classes with respect to right-equivalence. A {\bf normal form} for $K$ is given by a smooth\footnote{
That is, infnitely often differentiable.}  map
\[\Phi:\mathcal{B}\longrightarrow \C[x_1,\ldots,x_n]\subset\C\{x_1,\ldots,x_n\}\]
of a finite-dimensional $\C$-linear space $\mathcal{B}$ into the space of polynomials for which the following three conditions hold:
\begin{itemize}[leftmargin=10mm]
\item[(1)] $\Phi(\mathcal{B})$ intersects all equivalence classes of $K$,
\item[(2)] the inverse image in $\mathcal{B}$ of each equivalence class is finite,
\item[(3)] $\Phi^{-1}(\Phi(\mathcal{B})\setminus K)$ is contained in a proper hypersurface in $\mathcal{B}$.
\end{itemize}
The elements of the image of $\Phi$ are called {\bf normal form equations}.

\noindent We say that a normal form is a {\bf polynomial normal form} if the map $\Phi$ is  polynomial.
\end{defn}

\begin{example}
Referring to Arnold's list of normal forms, $f=x^4+y^4$ is of complex type $X_9$. The $\mu$-constant stratum of $f$ is covered by the complex normal form $\Phi: \C\to\C[x,y]$, $\Phi(a)=x^4+ax^2y^2+y^4$. For example, the germ $g=x^4+\epsilon x^3y+y^4$, for given $\epsilon$, is in the $\mu$-constant stratum of $f$. Hence, there exists a $\C$-algebra automorphism $\phi_1$ such that $\phi_1(g)= x^4+ax^2y^2+y^4$, for some $a\in\C$. In this case there also exists a $\C$-algebra automorphism $\phi_2$ such that $\phi_2(g)= x^4-ax^2y^2+y^4$.

Those $a\in \C$ for which $x^4+ax^2y^2+y^4$ is not in the $\mu$-constant stratum of $f$ are contained in a proper hypersurface. In the case of $X_9$, this hypersurface is given by the equation $a^2-4=0$. Precisely for these values of $a$, the Milnor number of $x^4+ax^2y^2+y^4$ is infinite, that is, the singularity is not isolated.
\end{example}

\begin{defn} (\cite{AVG1985})
The \textbf{modality} of a germ $f\in\m^2\subset \mathbb C\{x,y\}$ is the least number such that a sufficiently small neighborhood of $\jet(f,k)$, with $k$ an upper bound on the determinacy of $f$, can be covered by a finite number of $m$-parameter families of orbits under the right-equivalence action on the $k$-jet space. 
\end{defn}

A general question is how the number of parameters in a normal form relates to the modality. In the cases where Arnold has given a normal form, the two numbers are known to agree (this follows e.g. from the results in \cite{AVG1985}). In the subsequent section, for each corank $2$ singularity with non-degenerate Newton boundary, we give a normal form covering its $\mu$-constant stratum. We will see, in particular, that for corank $2$ singularities with non-degenerate Newton boundary,
 the number of parameters in the normal form equals the modality.

\section{Normal forms for Isolated Singularities with a non-degenerate Newton Boundary}
\label{section:NormalForm}

Arnold has constructed normal forms for the $\mu$-constant strata of (in particular) all singularities of corank $2$ and modality $\leq 2$. He has also associated a type to each $\mu$-constant stratum of these singularities (then also referred to as the type of a germ in the stratum). Restricting to the non-degenerate cases of corank $2$ and modality $\leq 2$, the type of a $\mu$-constant stratum can be labeled by the constant Newton polygon of the respective normal form covering the stratum (noting that all of Arnold's normal forms have different Newton polygons).

Taking up this pattern in the more general setting of any corank 2 germ with non-degenerate Newton boundary, our goal in this section is to show that the $\mu$-constant stratum of each such germ can be covered by a single normal form, and to explicitly specify such a normal form. Without loss of generality we can restrict to convenient germs. Our argument is based on the concept of an unfolding, and makes use of  \cite{GLS2007} and \citet{BGM2011}. 

\begin{remark}
In a fixed normal form, all normal form equations for different values of the parameters corresponing to points in the $\mu$-constant stratum under consideration have the same Newton polygon. Except for the Newton polygon, our normal form only depends on a choice of a regular basis on and above the Newton polygon for a specifically chosen special fiber. 
If two $\mu$-constant strata have normal forms with the same Newton polygon, they contain the same special fiber, hence must be equal. While a $\mu$-constant stratum can have normal forms with different Newton polygons, we will see that all possible Newton polygons are related via right-equivalence, and we can impose a normalization condition to fix one of them. So the type of a $\mu$-constant stratum  
can be labeled by the (up to normalization unique) Newton polygon of the normal form.
\end{remark}

\begin{defn}
A convergent power series $F\in\mathbb C\{{ x,t}\}=\mathbb C\{x_1,\ldots,x_n,t_1,\ldots,t_k\}$ is called an {\bf unfolding} of $f\in\mathbb C\{x_1,\ldots,x_n\}$ if $F({ x}, { 0})=f({ x})$. After choosing a representative  $U\times T\to\mathbb C$ for $F$ (again denoted by $F$) with open neighbourhoods of the origin $U\subset \mathbb C^n$ and $T\subset \mathbb C^k$ we use the notation \[F_{ t}({ x})=F({ x,t}),
\]
for the corresponding family $F_{ t}$ of holomorphic functions (or convergent power series) parametrized by $t\in T$. The neighborhood $T$ is called the base of $F$.
 
\end{defn}

\begin{defn}
Two unfoldings $F_1:U\times T\to\mathbb C$ and $F_2:U\times T\to\mathbb C$ 
are equivalent if one is taken into the other under the action of right-equivalence, that is, under the action of the Lie group $G$ of $\mathbb C$-algebra automorphisms from $\mathbb C\{x_1,\ldots,x_n\}$ to itself, that smoothly depends on $t\in T$. In other words, there exist a smooth map-germ $g:(T,0)\to(G,e)$, where $e$ is the identity element of $G$, such that \[F_1(t)=g(t)F_2(t).\]
\end{defn}

\begin{defn}
Let $F_1:U\times T\to \mathbb C$ and $F_2:U\times T'\to\mathbb C$ be unfoldings of $f\in\mathbb C\{x_1,\ldots,x_n\}$. Then $F_2$ is induced by $F_1$, if there exists a smooth map $\theta: (T',0)\to (T,0)$ such that the pullback of $F_1$ is $F_2$, that is $F_2=(\id_U,\theta)^*(F_1)$.
\end{defn}

\begin{defn}
An unfolding $F$ of $f\in\mathbb C\{x_1,\ldots,x_n\}$ is called \textbf{versal\,}\footnote{Let $f\in\m^2\subset \mathbb C\{x,y\}$ and $e_1,\ldots,e_{\mu(f)+1}$ be a basis of $Q_f$ and $\lambda=(\lambda_1,\ldots,\lambda_{\mu(f)+1})$ then $\F_\lambda(x,y)=f+\sum\lambda_ie_i$ is a versal unfolding.} if every unfolding of $f$ is equivalent to one induced by $F$. A versal unfolding for which the base has the smallest possible dimension is called \textbf{miniversal}. An unfolding preserving the Milnor number is called an \textbf{equisingular} unfolding.
\end{defn}

\begin{defn}\citep[\S 3, Definition 2]{G1974}\label{def:properModality}
Let $f\in\m^2\subset\mathbb C\{x,y\}$ be a germ of modality $m$ and let $\F_t(x,y)$ be a miniversal unfolding. The \textbf{proper modality} of $f$ is the dimension at the origin of the set of values $t$  for which $\F_t(x,y)$ has a singular point with Milnor number $\mu(f)$.
\end{defn}

In fact, we have:

\begin{theorem}\citep[\S 3, Theorem 6]{G1974} The modality and the proper modality of a germ $f\in\m^2\subset\mathbb C\{x,y\}$ coincide.
\end{theorem}

\begin{defn}\label{def:innermodality}\citep{A1974} For a germ $f\in\m^2\subset\mathbb C\{x,y\}$ with a non-degenerate Newton boundary, the \textbf{inner modality} is the number of all monomials in a regular basis for $Q_f$ lying on or above $\Gamma(f)$.
\end{defn}

We now connect Definitions \ref{def:properModality} and \ref{def:innermodality}.
In \cite{GLS2007}, Corollary  2.71 gives a local description of the $\mu$-constant stratum of a germ with a non-degenerate Newton boundary with respect to contact equivalence. The following result, which gives an equivalent version in terms of right equivalence, is an easy consequence thereof.
\begin{theorem}\label{corollary:unfolding}
Let $f\in\mathbb C\{x,y\}$ be a germ with a non-degenerate Newton boundary at the origin, then a miniversal, equisingular unfolding is given by 
\[F(x,y,{ t})=f+\sum_{i=1}^{m}t_ig_i,\]
where 
$m$ is the modality of $f$, and
 $g_1,\ldots,g_m$ represent a regular basis for $Q_f$ on and above $\Gamma(f)$.
\end{theorem}

It follows from the above result that:

\begin{theorem}
For a germ $f\in\m^2\subset\mathbb C\{x,y\}$ with a non-degenerate Newton boundary the proper modality and inner modality coincide.
\end{theorem}

The following general observation on the behaviour of bases of the local algebra in families is made in \citep{A1974}.

\begin{lemma}\label{lemma:regularBasis}
Assume that a family of smooth function germs $f$ depending smoothly on a finite number of parameters has $0$ as an isolated singularity with the same Milnor number $\mu$ for all values of the parameters of the family. Then every basis of the local algebra $Q_f$ of the function corresponding to the value $0$ of the parameter remains a basis for nearby values of the parameter.
\end{lemma}

\begin{proof}
Let $A=\mathbb C\{t\}$, $t=(t_1,\ldots,t_n)$, let $F\in R=A\{x,y\}$ define the family, write $f_t(x,y)=F(x,y,t)$, and let $M(t):=R/\Jac(f_t)$. By assumption, the Milnor number $\mu(t):=\operatorname{dim}_{\mathbb{C}}(M(t))$ of the germ defined by $f_t$ is finite. Let $m_1,\ldots,m_k\in R$ be monomials inducing a basis of $M(0)$. By Theorem \ref{thm fin det}, the $M(t)$ are locally finitely presented, hence form a coherent sheaf $\mathcal{M}$ on $\operatorname{Spec}(A)$. By the Lemma of Nakayama, the monomials $m_1,\ldots,m_k$ forming a basis of the fiber $M(0)$ extend to sections of $\mathcal{M}$ forming a generating system of the stalk at $0$. These sections are defined in 
a Euclidean open neighborhood of $0$. 
Since $\mathcal{M}$ is coherent, the sections also induce a generating system of every fiber $M(s)$ for $s\in U$. Since $\mu$ is assumed to be constant, $m_1,\ldots,m_k$ form a basis of $M(s)$ for $s\in U$.\footnote{We can formulate the proof in a more elementary way: Since $\mathcal{M}$  is locally finitely presented, that is, the cokernel of a matrix (which varies smoothly in terms of $t$), the monomials $m_1,\ldots,m_k$ (which do not vary in terms of $t$) form at $t=0$ a basis of a transversal space of the image of the presentation matrix, hence, they stay locally a generating system of a transversal space.}

\end{proof}

The proof of the lemma in particular yields the well-known fact, that the Milnor number is upper semi-continuous in families, that is, $\mu(s)\le \mu(0)$. In the case of germs with non-degnerate Newton boundary, we can refine the argument showing that a regular basis stays regular:

\begin{prop}\label{prop:regularBasis}
Let $f_0$ be a 
germ with a non-degenerate Newton boundary 
$\Gamma(f_0)$ and let $f$ be a  
germ with the same Newton polygon as $f_0$ and non-degenerate Newton boundary. Then for $f$ sufficiently close\footnote{Sufficiently close refers to the Euclidean distance in the $(\mu+1)$-jet space.} to $f_0$, the monomials in $\Mon(x,y)$ representing a regular basis for $f_0$ with respect to the  filtration defined by $\Gamma(f_0)=\Gamma(f)$ also represent a regular basis for $f$ with respect to the same filtration.
\end{prop}

\begin{proof}
Let $\mu=\mu(f_0)$, $w=w(f_0)$, $A=\mathbb C[t]$, $t=(t_1,\ldots,t_n)$, and let $F\in A\{x,y\}$ be an unfolding of $f_0$ covering the subspace of germs with Newton polygon $\Gamma(f_0)$ of the $(\mu+1)$-jet space, and write $f_t(x,y)=F(x,y,t)$. For the $A$-module $E^w_D$ generated by the monomials $m$ of $w\dash\deg(m)\geq D$ in the variables $x,y$, consider the vector spaces
\[M_D(t):=E^w_D/(E^w_D\cap \Jac(f_t)+E^w_{>D}).\] 
and let $m_1,\ldots,m_k\in E^w_D$ be monomials inducing a basis of $M_D(0)$. The $M_D(t)$ form a coherent sheaf $\mathcal{M}_D$ on $\operatorname{Spec}(A)$. 
Arguing as in the proof of Lemma \ref{lemma:regularBasis}, the sections also induce a generating system of every fiber $M_D(s)$ for $s\in U$. In particular, we have $\dim_\mathbb{C} M_D(s)\le \dim_\mathbb{C}M_D(0)$. On the other hand, \[\bigoplus_DM_D(s)=\mathbb C\{x,y\}/\Jac(f_s) \text{ and } \dim\bigoplus_DM_D(s)=\mu=\dim\bigoplus_DM_D(0),\]which implies that $\dim_\mathbb{C}M_D(0)=\dim_\mathbb{C}M_D(s)$ for all $D$. Therefore a regular basis for $f_0$ is a regular basis for $f_s$.

\end{proof}

As a next step in finding a normal form for the $\mu$-constant stratum of a germ with a non-degenerate Newton boundary, we will observe that the topological type of all germs in the $\mu$-constant stratum is the same. We will conclude that all germs can be transformed to germs with the same Newton polygon, and that for each of the these germs the Newton boundary is non-degenerate. 

\begin{defn}
Let $f,g\in\mathbb C\{x,y\}$ be germs with singular points at $0$. Then $f$ is \textbf{topologically equivalent} to $g$ when there are neighbourhoods $U$ and $V$ of $0$ in $\mathbb C^2$ and a homeomorphism $\phi:U\to V $ with $\phi(0)=0$  such that for the zero sets $X$ of $f$ and $Y$ of $g$ we have $$\phi(X\cap U)=Y\cap V.$$ The \textbf{topological type} of $f$ is the equivalence class of $f$.
\end{defn}

\begin{theorem}\label{theorem:Brieskorn}(\citet{BK1986}, Theorem 15)
Let $f\in\mathbb C\{x,y\}$ and $g\in\mathbb C\{x,y\}$ be germs with singularities at $0$. Then $f$ and $g$ are topologically equivalent if and only if there exists a bijection between the irreducible branches of $f$ and $g$ (
that is, the factors of $f$ and $g$ in $\mathbb C\{x,y\}$) such that 
\begin{itemize}[leftmargin=10mm]
\item the corresponding characteristic exponents are the same, and
\item the intersection numbers between corresponding branches coincide.
\end{itemize} 

\end{theorem}

\begin{theorem} \label{theorem:Le}(\citet{LR1976}, Theorem 2.1)
Let $F({ x},t)$ be a family of $\mu$-constant polynomials in ${ x}=(x_1,\ldots,x_n)$, $n\neq 3$, with isolated singularity at $0$ and coefficients which are smooth complex valued functions of $t\in I=[0,1]$. Then the topological type of the singularities is constant in the family.
\end{theorem}

\begin{remark}
Since the $\mu$-constant stratum of a plane curve singularity is smooth, see \cite{Wahl} and  \cite[Theorem II.2.61, Corollary II.2.67]{GLS2007} for an alternative proof, the $\mu$-constant stratum can locally be parametrized by polynomial families  depending smoothly on a finite number of parameters.
Hence, by Theorem \ref{theorem:Le}, the topological type is constant on the $\mu$-constant stratum. By Theorem \ref{theorem:Brieskorn}, this implies that the characteristic exponents and intersection numbers of the Puiseux expansions  are constant on the $\mu$-constant stratum. 
\end{remark}

\begin{remark}\label{rmk sect 4 1}
If two germs $f_1$ and $f_2$ have the same characteristic exponents and intersection numbers, and $f_1$ has a non-degenerate Newton boundary, then $f_2$ is right-equivalent to a germ with non-degenerate Newton boundary and same Newton polygon as $f_1$. Moreover, any germ with maximal Newton number in the right-equivalence class of $f_2$ has a non-degenerate Newton boundary. Note that the converse is also true by Theorem \ref{theorem:milnorNumber}.
\end{remark}
For clarity of the presentation, the proof of this fact will be given in the next section in Proposition~\ref{prop non-deg} and Corollary~\ref{prop:PuiseuxIntersectionNewtonBoundary}. We conclude from the remark that, if the $\mu$-constant stratum contains a germ with non-degenerate Newton boundary, then every germ in the $\mu$-constant  stratum is equivalent to a germ with non-degenerate Newton boundary. 
In particular, Lemma \ref{corollary:unfolding} gives the existence of local polynomial families covering the whole $\mu$-constant stratum.
    Combining our observations we obtain:
\begin{theorem}\label{prop:newtonBoundary}
Let $f\in\mathbb C\{x,y\}$ be a convenient germ with non-degenerate Newton boundary~$\Gamma$. Then all the germs in the $\mu$-constant stratum of $f$ are equivalent to a germ with the Newton polygon $\Gamma$ and  non-degenerate Newton boundary.
\end{theorem}

We use the following result of \citet{BGM2011} to prove the main result of the section.

\begin{prop}(\cite{BGM2011}, Corollary 4.6)\label{corollary:Markwig}
Let $f\in\mathbb C\{x,y\}$ be a convenient germ with a non-degenerate Newton boundary. Let $f_0$ be the principal part of $f$ and let $\{e_1,\ldots,e_n\}$ be the set of all monomials in a regular basis for $f_0$ lying above $\Gamma(f_0)$. Then there are $\alpha_i$ such that

\[f\sim f_0+\sum_{i=1}^{n} \alpha_ie_i.\]

\end{prop}

We now give the main result of this section:

 \begin{theorem}\label{theorem:normalform}
 Let $f$ be a convenient germ with non-degenerate Newton boundary. Let $f_0$ be the sum of the terms of $f$ lying on the vertex points of  $\Gamma(f)$, and let  $\{e_1,\ldots,e_n\}$ be the set of all monomials in a regular basis for $f_0$ lying on or above $\Gamma(f)$.  Then the family
 \[f_0+\sum_{i=1}^{n} \alpha_ie_i,\]
is a normal form of the $\mu$-constant stratum containing $f$. 
Restricting the parameters $\alpha_1,\ldots,\alpha_n$ to values such that every germ $f_0+\sum_{i=1}^{n} \alpha_ie_i$ has a  non-degenerate Newton boundary and the same Newton polygon as that of $f$, we obtain all germs in the $\mu$-constant stratum of $f$.  \end{theorem}
\begin{proof}
We first show that every germ $g$ in the $\mu$-constant stratum of $f$ is equivalent to a germ in the family $f_0+\sum_{i=1}^{n} \alpha_ie_i$. Note that by Theorem \ref{prop:newtonBoundary}, every germ in the $\mu$-constant stratum of $f$ is equivalent to a germ with the same Newton polygon as $f$ and non-degenerate Newton boundary.  

Let  $B$ be a regular basis of $f_0$. By Proposition \ref{prop:regularBasis}, $B$ is also a regular basis for every 
germ with the same Newton boundary in a sufficiently small neighborhood $M'$ of $f_0$ in the $\mu$-constant stratum. By Theorem  \ref{corollary:unfolding}, there exists a Euclidean open neighborhood $M''$ of $f_0$ contained in the $\mu$-constant stratum of $f_0$ such that all the germs in  $M''$ are equivalent to at least one of the germs in the family $f_0+\sum_{i=1}^{n} \alpha_ie_i$. Let $M:=M'\cap M''$. 

Since $g$ is equivalent to a germ with the same Newton polygon as $f$ and with non-degenerate Newton boundary, we may assume without loss of generality that $\Gamma(g)=\Gamma(f)$.  
We denote the sum of the terms of $g$ on $\Gamma(g)$ by $g_0$. 

As a linear combination of all monomials on the Newton boundary of $f_0$, we obtain a germ $f'\in M$ with the same monomials as $g_0$ and with regular basis $B$. Hence, by varying coefficients of the terms on the Newton polygon, we can construct a path $\tau$ in the germ space from $g_0$ to $f'$ with constant Newton polygon. Since being degenerate is a Zariski closed condition, there are only finitely many points on the path such that the Newton boundary is degenerate. Moreover, if we denote by $\reg(h)$ the set of all regular bases of the germ $h$, then for all germs $h$ in  a Zariski open subset $U\subset \tau$, the set $\reg(h)$ is the same. Since, by Proposition \ref{prop:regularBasis}, $\reg(h)$ is constant on a Euclidean open neighborhood of $g_0$, we have $g_0\in U$.  Hence, the set $A:=\{h\in \tau |\ h \text{ Newton non-degenerate, }\reg(h) =\reg(g_0) \}\subset\tau$ is Zariski open. By construction, $\tau \cap M$ is contained in the subspace of $M$ of germs with the same Newton polygon as $f_0$ (and with regular basis $B$). This implies that the intersection $A \cap  \tau \cap M$ is non-empty, so $B\in\reg(g_0)$.

Hence, writing $e_1,\ldots, e_n$ for the elements of $B$ above the Newton polygon and $e_{n+1},\ldots,e_{r}$ for the elements of $B$ on the Newton polygon\footnote{Note that there are also elements in $B$ under the Newton polygon, which are not relevant here.}, Proposition \ref{corollary:Markwig} implies that $g\sim g_0 + \sum_{i=1}^n \alpha'_i e_i$ with some $\alpha'_i$. We will now show that $g_0+ \sum_{i=1}^n \alpha'_i e_i\sim f_0 + \sum_{i=1}^r \alpha_i e_i$ with some $\alpha_i$, which proves our claim. In fact, we will show that, after modifying $g$ without loss of generality by a right-equivalence keeping the Newton polygon of $g$, there are $\alpha''_i$ with $g_0=f_0 + \sum_{i=n+1}^r \alpha''_i e_i$. 

By the second part of Remark \ref{rmk sect 4 1} (which refers to Proposition \ref{prop non-deg} in the next section), any germ with maximal Newton number in the
right-equivalence class of a germ with non-degenerate Newton boundary also has a non-degenerate Newton boundary. So if $g'$ is right-equivalent to $g$ with $\Gamma(g')=\Gamma(g)$, then $N(g')=N(g)=\mu(g)=\mu(g')$ (using Theorem \ref{theorem:milnorNumber}), so $g'$ has maximal Newton number, and thus a non-degenerate Newton boundary. Hence, a  right-equivalence keeping the Newton polygon also keeps non-degeneracy of the Newton boundary.

Write $w:=w(f_0)=w(g_0)$ and $d:=d(f_0)=d(g_0)$. The only elements of $\Jac(f_0)$ which have terms of $w$-degree $d$ are of the form $c_1x\frac{\partial f_0}{\partial x}$, $c_2y\frac{\partial f_0}{\partial y}$, $c_3y^{n_1}\frac{\partial f_0}{\partial x}$ and $c_4x^{n_2}\frac{\partial f_0}{\partial y}$ or are linear combinations thereof.  We can make a connection between terms occuring in a basis of $E_{d}^w/(\Jac(f_0)+E_{>d}^w)$ and terms occuring in $g$ by the following argument:
\begin{enumerate}[leftmargin=7mm]
\item If there is an $n_1$ such that $\supp(c_3y^{n_1}\frac{\partial f_0}{\partial x})$ contains a monomial $x^\alpha y^\beta$ of $w$-degree $d$ (note that there can be at most one such monomial), the corresponding term in $g$ can be modified by a right equivalence of the form $x\mapsto x+ cy^{n_1}$, $y\mapsto y$, possibly generating terms of higher $w$-degree. We distinguish two cases:
\begin{enumerate}[leftmargin=7mm]
 \item If $\alpha\neq 0$, then  $x^\alpha y^\beta$ is not a vertex monomial (a monomial corresponding to vertices of the Newton polygon): Removing a vertex monomial  by the above transformation would increase the area under the Newton polygon and thus, according to Theorem~\ref{theorem:milnorNumber}, change the Milnor number.
 \item If $\alpha=0$, then via the above transformation we may assume that the coefficient of $y^\beta$ in $g$ is normalized to $1$.
 \end{enumerate} 
\item If there is a monomial of $c_4x^{n_2}\frac{\partial f_0}{\partial y}$ of $w$-degree $d$, we can agrue in the same way as in (1), using a right equivalence of the form $y\mapsto y+ cx^{n_2}$, $x\mapsto x$.
\item Depending on whether  $x\frac{\partial f_0}{\partial x}$ and $y\frac{\partial f_0}{\partial y}$ are $\mathbb{C}$-linearly independent in $E_{d}^w$ or not, the elements $x\frac{\partial f_0}{\partial x}$ and $y\frac{\partial f_0}{\partial y}$ of $\Jac(f_0)$ can remove one or two vertex monomials in a basis of $E_{d}^w/E_{>d}^w$.  After modifying $g$ by a suitable right equivalence of the form $x\mapsto ax$, $y\mapsto by$ (which again does not change the Newton polygon of $g$), we may assume that the (one or two) vertex monomials have coefficients normalized to $1$.
\end{enumerate}

We can, hence, conclude that every monomial of $w$-degree $d$ is either in $B$ or does not appear in $g$ or is a vertex monomial  of $g$ with coefficient $1$. Hence, as claimed, there are $\alpha''_i$, with $g_0=f_0 + \sum_{i=n+1}^r \alpha''_i e_i$.\bigskip

We now show that the exceptional locus in the family is given by a hypersurface in the base space, and after removing this hypersurface, we still obtain up to right-equivalence all germs in the $\mu$-constant stratum: 
First note, that if we restrict $\alpha_1,\ldots,\alpha_n$ to values such that every germ $f_0+\sum_{i=1}^{n} \alpha_ie_i$ has a  non-degenerate Newton boundary and Newton polygon $\Gamma(f)$, by Theorem~\ref{theorem:milnorNumber}, we obtain only germs in the $\mu$-constant stratum of $f$.
Moreover, this restriction amounts to considering parameter values in the complement of a hypersurface, which is the union of the hypersurfaces where the vertex terms vanish and the discriminant conditions for the individual faces. On the other hand, none of the germs corresponding to parameter values in the hypersurface is in the $\mu$-constant stratum of $f$: Vanishing of a vertex monomial, by Theorem \ref{theorem:milnorNumber}, leads to a larger Milnor number. Assume  the Newton boundary of a germ $f_1 = f_0+\sum_{i=1}^{n} \alpha_ie_i$ becomes degenerate (with no vertex monomial is vanishing), and the germ is right-equivalent to a germ $f_1'$ with non-degenerate Newton boundary and Milnor number $\mu(f_1')=\mu(f)$. Then, by Theorem \ref{prop:newtonBoundary}, we may assume that $\Gamma(f_1')=\Gamma(f_0)$. Hence, we obtain $\mu(f_1)=\mu(f_1')=N(f_1')=N(f_0)=N(f_1)$. In particular $f_1$ has maximal Newton number in its right-equivalence class. Then, as above, 
the second part of Remark \ref{rmk sect 4 1} implies that $f_1$ is non-degenerate, a contradiction.
\bigskip

Finally, we show that our family $f_\alpha = f_0+\sum_{i=1}^{n} \alpha_ie_i$ meets every right-equivalence class in only finitely many points. 
Fix a parameter value $\beta$. Then for any parameter value $\gamma$ with $f_{\gamma} \sim f_{\beta}$, by \cite[Part II, Proposition 2.48(i)]{GLS2018} it follows that in a sufficiently small neighborhood of $\gamma$ in the parameter space, the set of all $\gamma'$ with $f_{\gamma'} \sim f_{\beta}$ is closed. Since the parameter space is paracompact, it admits a locally finite cover of such neighborhoods. Hence, the set of all $\gamma'$ in the parameter space with $f_{\gamma'} \sim f_{\beta}$ is closed. So, if we assume that there are infinitely many such~$\gamma'$, there exists a non-constant one parameter subfamily $f_t = f_0+\sum_{i=1}^{n} \alpha_i(t)e_i$ with $f_t \sim f_\beta$. Then, by \cite[Theorem 9.1.5]{PdJ2000}, it follows that there is an $e_i\in \Jac(f)$, a contradiction.
\\
\end{proof}

\begin{remark}
Theorem \ref{theorem:normalform} also is true in the formal sense.
\end{remark}
\begin{proof}
Any formal power series $f$ with finite Milnor number is (via a formal transformation) right-equivalent to a polynomial (see Remark \ref{rem Artin}). The first part of the proof of Theorem \ref{theorem:normalform} applied to this polynomial shows that $f$ is right-equivalent to a germ in the family $\operatorname{NF}(f)$.  In order to see that this family intersects every right-equivalence class of formal power series only in finitely many points, note that, by the second part of the proof of Theorem \ref{theorem:normalform}, it intersects the corresponding right-equivalence class of convergent power series  in only finitely many points and that all germs in $\operatorname{NF}(f)$ are polynomials, and hence are convergent. The claim then follows by  the Artin approximation theorem.
\end{proof}

We finish the current section by discussing in which sense  the normal form associated  by Theorem \ref{theorem:normalform} to a germ which is equivalent to a germ with non-degenerate Newton boundary is unique, and how it can be used to label $\mu$-constant strata. We first observe:

\begin{remark}
Any  $\mu$-constant stratum containing a germ with a non-degenerate Newton boundary has a normal form as constructed in Theorem \ref{theorem:normalform} with fixed Newton polygon. Any such normal form depends only on the vertices of the Newton polygon $\Gamma$ and the choice of a regular basis. For fixed Newton polygon $\Gamma$, different choices of regular bases lead to normal forms of the same $\mu$-constant stratum.
However, for the same $\mu$-constant stratum there can be normal forms with different non-degenerate Newton boundaries. Hence, like Arnold, we associate a \textbf{type} to each $\mu$-constant stratum and then make a specific choice of a normal form for each type $T$. This choice amounts to fixing a Newton polygon and a regular basis of monomials on and above the Newton polygon. For distinguishing between different types, it is sufficient to know the Newton polygon of the normal form. \end{remark}

To achieve uniqueness of the Newton polygon associated to a fixed type (in order to label types by Newton polygons), the Newton polygon may be chosen according to the following normalization condition, which 
can be satisfied via  a right-equivalence for any germ with non-degenerate Newton boundary. Moreover, the normalization condition is satisfied for all germs in the  normal form produced by Theorem~\ref{theorem:normalform}

  \begin{remark}\label{remark:NB}
  We will refer to a weighted homogeneous non-unit in $\mathbb C[x,y]$ which is linear in one variable as \textbf{linear homogeneous}, otherwise we refer to it as non-linear homogeneous. 

Let $f\in\m^2\subset\mathbb C\{x,y\}$ be a convenient germ with non-degenerate Newton boundary, and let $\Delta$ be a facet of $\Gamma(f)$, and write $w=w(\Delta)$.
If in the factorization of $\jet(f,\Delta)=w\dash\jet(f,d(f))$ in $\mathbb C[x,y]$ would occur any linear homogeneous factor except $x$ or $y$ with exponent larger than one, 
or if it has a non-linear factor of multiplicity larger than one, then $w\dash\jet(f,d(f))$ would have degenerate saturation. Hence, we can factorize $w\dash\jet(f,d(f))$ in $\mathbb C[x,y]$ as $$w\dash\jet(f,d(f))=x^a \cdot y^b 
\cdot g_1\cdots g_n
\cdot \widetilde{g},$$ where $a$,$b$ 
are integers, $g_1,\ldots,g_n$ are linear homogeneous polynomials not associated to $x$ or $y$
, and $\widetilde{g}$ is a product of non-associated irreducible non-linear homogeneous polynomials. If $w(x)=w(y)$, we can change $f$ via a right-equivalence such that for the resulting germ $a\neq 0$ and $b\neq 0$. If $w(x)>w(y)$, we can change $f$ such that for the resulting germ $a=0$ only if $n=0$. If $w(x)<w(y)$, we can change $f$ such that for the resulting germ $b=0$ 
only if $n=0$.    

 \begin{proof}   Note that if $w(x)>w(y)$ and $a=0$ (i.e., $\Delta$ meets the $y$-axis),  then for $n\neq 0$ there is a right-equivalence with $g_1\mapsto x$, $y\mapsto y$. The case $w(x)< w(y)$, $b=0$ can be treated in a similar way. The same applies if $w(x)=w(y)$ for the cases $a=0$, $b\neq 0$ and $a\neq 0$, $b=0$. If $a=b=0$, then by $f\in \m^2$ we have $n\geq 2$, hence there is a right-equivalence with $g_1\mapsto x$, $g_2\mapsto y$.
\end{proof}

\begin{defn}
\label{def normal cond}
Let $f\in\m^2\subset\mathbb C\{x,y\}$ be a convenient germ with non-degenerate Newton boundary, and let $\Delta$ be a facet of $\Gamma(f)$, and write $w=w(\Delta)$. Then $\jet(f,\Delta)$ factorizes in $\mathbb C[x,y]$ as $$\jet(f,\Delta)=x^a \cdot y^b 
\cdot g_1\cdots g_n
\cdot \widetilde{g},$$ where $a$,$b$ 
are integers, $g_1,\ldots,g_n$ are linear homogeneous polynomials not associated to $x$ or $y$, and $\widetilde{g}$ is a product of non-associated irreducible non-linear homogeneous polynomials. We say that $f$ is \textbf{normalized} with respect to the facet $\Delta$, if 
\[%
\begin{tabular}
[c]{ccccc}%
$w(x)=w(y)$ &  &  & $\Longrightarrow$ & $a,b\neq0$\\
$w(x)>w(y)$ & \text{and} & $a=0$ & $\Longrightarrow$ & $n=0$\\
$w(x)<w(y)$ & \text{and} & $b=0$ & $\Longrightarrow$ & $n=0$%
\end{tabular}
\]

\end{defn}

This observation can be used to choose a natural normalization for germs with non-degenerate Newton boundary:
\begin{lemma}\label{lem normal}
There exists a right-equivlance such that $f$ is normalized with respect to all facets of the Newton polygon. \end{lemma}
\begin{proof} For any facet not meeting the coordinate axes, the normalization condition is empty since both $a$ and $b$ are non-zero. The claim follows then by the following iteration: If there is only one facet meeting a coordinate axis (hence by convenience both axes), we apply Remark \ref{remark:NB}. If there are two facets $\Delta_1$ and $\Delta_2$ meeting the $x$- and $y$-axis, respectively, then by  Remark~\ref{remark:NB} there are right-equivalences $\phi_1$ and $\phi_2$ which normalize the jets of $\Delta_1$ and $\Delta_2$ individually. Applying the composition of $\phi_1$ and $\phi_2$ in any order will normalize the jets of both facets: Without loss of generality, we may assume that $\phi_1(x)=x,\, \phi_1(y)=y+c_1\cdot x^{n_2}$ and $\phi_2(x)=x+c_2\cdot y^{n_2},\, \phi_2(y)=y$ with $c_1,c_2\in \mathbb{C}$. Hence, by convexity of the Newton polygon, $\phi_1$ does not change the jet of any facet except $\Delta_1$, analogously $\phi_2$ does not change the jet of any facet except $\Delta_2$. Truncating at the determinacy and repeating the process either normalizes (after a single iteration) or deletes (after possibly several iterations) facets meeting a coordinate axis. The process, hence, determines in finitely many steps a germ which is normalized with respect to all faces (but may not be convenient).
\\
\end{proof}

Since the definitions and results in the standard literature are typically formulated for convenient germs, it makes sense to relate our observations to the setting of convenient germs. To do so, we define:

\begin{defn}
  A facet of the Newton polygon of a germ is called a \textbf{smooth face} if the saturation of its jet is smooth.
  \end{defn}

\begin{remark}
 A germ $f$ which is normalized with respect to every facet of the Newton polygon may not be convenient. However, adding the terms $x^{d}$ or $y^{d}$ with $d=\mu(f)+2$, if necessary, yields a convenient germ which is right-equivalent to $f$. In this way, we obtain a convenient germ, which is normalized with respect to all facets of the Newton polygon except for smooth facets meeting a coordinate axis. If a germ with non-degenerate Newton boundary satisfies this property, we say that it satisfies the \textbf{normalization condition}.\footnote{Note that a germ satisfying the normalization condition may or may not have smooth facets meeting a coordinate axis, that is, there may exist non-smooth facets meeting a coordinate axis. 
  }

\end{remark}

\begin{remark}
For any germ which is equivalent to a germ with non-degenerate Newton boundary, the germs with non-degenerate Newton boundary in its right-equivalence class which satisfy the normalization condition have a unique Newton polygon (up to permuation of $x$ and $y$).
\end{remark}

This follows directly from the following remark, since characteristic exponents and intersection numbers are invariant under right-equivalence. 
\begin{remark}\label{rmk sec4 b}
If  a germ is right-equivalent to a germ with non-degenerate and normalized Newton polygon $\Gamma$, then the characteristic exponents and intersection numbers of the branches of $f$ uniquely determine $\Gamma$ (up to permutation of the variables). 
\end{remark}
For clarity of the presentation, the proof will be given in the next section in Corollary~\ref{prop:PuiseuxIntersectionNewtonBoundary2}. More generally, combining Remark \ref{rmk sec4 b} with  Remark \ref{rmk sect 4 1}, and Lemma \ref{lem normal}, we obtain:

 \begin{theorem}In a $\mu$-constant stratum which contains a germ with a non-degenerate Newton boundary,  every right-equivalence class contains a normalized germ, and all germs in the $\mu$-constant stratum  satisfying the normalization condition have the same Newton polygon (up to permutation of the variables).
 \end{theorem} 

Throughout the paper we will choose normal forms according to the normalization condition, which means that any germ in the image of the normal form which is in the $\mu$-constant stratum satisfies the normalization condition.   
 We denote the chosen normal form corresponding to the type $T$ by $\NF(T)$. 
   For $b\in \parm(\NF(T)):=\Phi^{-1}(K)$ with $K$ as in Definition~\ref{def nfequ}, we write $\NF(T)(b):=\Phi(b)$ for the corresponding normal form equation.
\end{remark}
 
Definition \ref{def NB} extends then to types:
 \begin{defn}
 Let $T$ be a singularity type of a germ with a non-degenerate Newton boundary. We call
 \begin{eqnarray*}
 \supp(T)&:=&\supp(\NF(T)(b))
 \end{eqnarray*}
where $b\in\parm(\NF(T))$ is generic, the \textbf{support} of $T$. In the same way, we define the \textbf{Newton boundary} $\Gamma(T)$ of the type, and the notations $\Gamma_+(T)$, $d(T)$, and $\supp(T,\Delta)$. If for a given type $T$, $w\dash\jet(\NF(T)(b),j)$ is independent of $b\in\parm(\NF(T))$, we denote it by $w\dash\jet(T,j)$. 
 \end{defn}

\section{The correspondence between Characteristic Exponents, Intersection Numbers and normalized non-degenerate Newton boundaries}\label{sec char exp}
In this section, we will prove the results stated in  Remarks \ref{rmk sect 4 1} and  \ref{rmk sec4 b}. So we will show that, if two germs have the same characteristic exponents and intersection numbers, and one of them has a non-degenerate Newton boundary, then also the other one is right-equivalent to a germ with non-degenerate Newton boundary and same Newton polygon. Moreover, we will prove that if a germ $f$ is equivalent to a germ with a non-degenerate Newton boundary, this Newton boundary is uniquely determined by the characteristic exponents and intersection numbers of the branches of $f$, provided we require the normalization condition as specified in Definition \ref{def normal cond}. The results will be given in Corollaries \ref{prop:PuiseuxIntersectionNewtonBoundary2} and \ref{prop:PuiseuxIntersectionNewtonBoundary}.\smallskip

We first recall some facts on Puiseux expansions, which play a key role in the proofs.
\begin{theorem} \cite[Theorem 5.1.1]{PdJ2000} \label{theorem:parametrization}
Let $0\neq f\in\mathbb C \{x,y\}$ be irreducible. Then there exist $x(t), y(t)\in\mathbb C\{t\}$ such that
\begin{itemize}[leftmargin=8mm]
\item[(i)] $f(x(t),y(t))=0$,
\item[(ii)] $\dim_{\mathbb C}(\mathbb C\{t\}/\mathbb C\{x(t),y(t)\})<\infty$.
\end{itemize}
The tuple $(x(t),y(t))$ is called a \textbf{parametrization} of $f$.
\end{theorem}
Parametrizations can be computed using the Newton-Puiseux algorithm.
 A germ $f\in\mathbb C\{x,y\}$ is called $y$-general, if $f$ is not divisible by $x$, and $f$ is called $y$-general of order $b$, if $f(0,y)$ has a zero of order $b$. The field of Puiseux series over $\mathbb C$ is
$\mathbb C\{\{x\}\}=\bigcup_{k=1}^{\infty}\mathbb C((x^{\frac{1}{k}}))$. The valuation ring $\mathcal{P}(x)=\mathbb C\{\{ x\}\}_{v\ge 0}=\bigcup_{k=1}^{\infty}\mathbb C \{x^{\frac{1}{k}}\}$ with respect to the canonical valuation $\nu$ consists of all Puiseux series with non-negative exponents only. By the Newton-Puiseux theorem (see, for example, \cite[Theorem I.3.3]{GLS2007}), any $y$-general germ $f\in \mathbb C\{x\}[y]$, can be factorized as $f=c\cdot \prod_{i=1}^d(y-\phi_i(x))$ with $c\in \mathbb C$ and the Puiseux expansions $\phi_i\in \mathcal{P}(x)$. 
By passing to an irreducible factor of $f$ and translating the singularity to the origin, it is enough to consider irreducible Weierstrass polynomials. In this case, the structure of the Puiseux expansions has a straight-forward description.

\begin{theorem} \label{irreducible Puisuex}\cite[Theorem 5.1.7]{PdJ2000}
Let $f
\in\mathbb C\{x\}[y]$ be an irreducible Weierstra\ss\ polynomial of $y$-degree $n$.
Let $\epsilon$ be a primitive $n$-th root of unity. 
 Then there exists a power series $\eta\in t \cdot\mathbb C\{t\}$ such that $f$ factorizes in $\mathbb C\{x^{\frac{1}{n}}\}[y]$ as
\[
f=\prod_{i=1}^n(y-\eta(\epsilon^ix^\frac{1}{n})).
\]
In particular, all Puiseux expansions of $f$ are contained in $\mathbb C\{x^\frac{1}{n}\}$.
\end{theorem}

A Puiseux expansion of a $y$-general germ can be computed using the Newton-Puiseux algorithm, which we recall for further reference as Algorithm \ref{alg:Puiseux}. Note that, by branching into all possible choices of facets $\Delta$ in the algorithm, one can obtain all Puiseux expansions of $f$.
\begin{algorithm}[h]
\caption{Newton-Puiseux algorithm}%
\label{alg:Puiseux}
\begin{spacing}{1.05}
\begin{algorithmic}[1]

\Require{A $y$-general germ $f\in\langle x,y \rangle \subset \mathbb C \{x\}[y]$ and a bound $l\geq 0$.}

\Ensure{A Puiseux expansion $\psi$ of $f$ up to order $l$}

\State $i:=0$, $f^{0}:=f$, $x_0:=x$, $y_0:=y$, $K_{(0)}=\mathbb C ((x_0))$  
\While{true}
\State Let $\Delta$ be a facet of $\Gamma(f^{(i)})$ with slope $-\frac{p_i}{q_i}$ in lowest terms \label{line:face}
\State $w:=(p_i,q_i)$
\State Let $a_i$ be a root of the univariate polynomial $w\dash\jet(f^{(i)},d_w)(1,y_{i+1})\in\mathbb C[y_{i+1}]$
\State \label{line:fieldExtension}With new variables $x_{i+1}$, $y_{i+1}$ and $s_{i+1}$, set
\begin{eqnarray*}
f^{(i+1)}&:=&\frac{1}{x_{i+1}^{d_w}}f^{(i)}(x_{i+1}^{p_i}, x_{i+1}^{q_i}(a_i+y_{i+1}))\in
\mathbb C \{x_{i+1}\}[y_{i+1}]
\\
s_i&:=&x_i^{\frac{q_i}{p_i}}(a_i+s^{(i+1)})\in K_{(i+1)}[s^{i+1}],
\end{eqnarray*}
\hskip\algorithmicindent where $K_{(i+1)}=K_{(i)}[x_{i+1}]/(x_{i+1}^{p_i}-x_i)$
\If{$f^{(i+1)}\in\mathbb C[y_{i+1}]$
or $\frac{q_0}{p_0}+\frac{q_1}{p_0p_1}+\ldots+\frac{q_i}{p_0\cdot\ldots\cdot p_i}\ge l$}
\State Replace successively $s^{(j+1)}$ in $s^{(j)}$ for $j=0,\ldots,i$ and obtain
\begin{eqnarray*}
s^{(0)}&:=&x^{\frac{q_0}{p_0}}(a_0+x_1^{\frac{q_1}{p_1}}(a_1+\ldots+x_i^{\frac{q_i}{p_i}}(a_i+s^{(i+1)})))\\
&=&x^{\frac{q_0}{p_0}}(a_0+x^{\frac{q_1}{p_0p_1}}(a_1+\ldots+x^{\frac{q_i}{p_0\cdot\ldots\cdot p_i}}(a_i+s^{(i+1)})))
\end{eqnarray*}
\hskip\algorithmicindent\hskip\algorithmicindent  and set $s^{(i+1)}=0$ in the expression.
\Return $s^{(0)}$
\Else
\State $i:=i+1$

\EndIf
\EndWhile
\end{algorithmic}
\end{spacing}
\end{algorithm}

\begin{remark}
\label{rmk height2}
\begin{enumerate}[leftmargin=7mm]
\item Starting out with the field 
$K_{(0)}= \mathbb C((x))$, this field is enlarged in the $i$-th iteration of the algorithm to 
$K_{(i+1)}=K_{(i)}(x^\frac{1}{p_0\cdot\ldots\cdot p_i})$. 

\item \label{item n} Suppose $f$ is $y$-general and $n$ is minimal such that all Puiseux expansions of $f$ are in $\mathbb C (x^{1/n})$. Setting $x(t)=t^n$ and $y(t)=\psi(t^n)$ with $\psi$ as returned by Algorithm \ref{alg:Puiseux}, we obtain a parametrization of $f$ as in Theorem \ref{theorem:parametrization} developed up to order $l$.

\item \label{lemma:height} If $f$ is irreducible and $y$-general of order $b$, then  
$n=b$ \cite[Proposition~I.3.4]{GLS2007}. 
\end{enumerate}
\end{remark}

\begin{defn} \cite[Definition 5.2.14]{PdJ2000} \label{def:characteristicExponents}
For $\gamma\in \mathcal{P}(x)$, let $n\in\mathbb N$ be minimal with $\gamma\in\mathbb C\{x^\frac{1}{n}\}$, and write $\gamma=\sum_{j\ge m}b_jx^\frac{j}{n}$, with coefficients $b_j\in\mathbb C$, $b_m\neq 0$. We define
\begin{eqnarray*}
k_0&:=&\min\{n,m\},\\
k_1&:=&\left\{
\begin{array}{lll}
\min\{j\ |\ b_j\neq 0,\ \gcd\{j,k_0\}<k_0\}&\quad&\text{if }n\le m,\\
\min\{j\ |\ b_j\neq 0,\ \gcd\{j,k_0\}<k_0\}+n-m&\quad&\text{if }n> m \text{ and }m\mid n,\\
n&\quad&\text{if }n> m \text{ and }m\nmid n,
\end{array}
\right.\\
&&\hspace{-1,5cm}\text{and for }\nu\ge2,\\
k_\nu&:=&\left\{
\begin{array}{lll}
\min\{j\ |\ b_j\neq 0,\ \gcd\{j,k_0,\ldots,k_{\nu-1}\}< \gcd\{k_0,\ldots,k_{\nu-1}\}&\quad&\text{if }n\le m,\\
\min\{j\ |\ b_j\neq 0,\ \gcd\{j,k_0\}<k_0\}+n-m&\quad&\text{if }n> m \text{ and }m\mid n,\\
n&\quad&\text{if }n> m \text{ and }m\nmid n,
\end{array}
\right.
\end{eqnarray*}
The number $k_\nu$ is called the $\nu$-th \textbf{characteristic exponent} of $\gamma$. Note that, by construction, there is a minimal $g$ such that $\{k_\nu \mid \nu \in \mathbb{N}_0\}=\{k_0,\ldots,k_g\}$. Also note that in all cases $k_0<k_1<\ldots<k_g$ and $\gcd(k_0,\ldots,k_g)=1$.

If there is a maximal $m_0\geq 0$ with $\sum_{j =m}^{m_0}b_jx^\frac{j}{n}\in \mathbb C\{x\}$, then we define the \textbf{rational part} of~$\gamma$  as the polynomial $\gamma^{rat}=\sum_{j =m}^{m_0}b_jx^\frac{j}{n}$, otherwise we set $\gamma^{rat}=\gamma$.
\end{defn}

\begin{remark}\label{remark:Puiseux}
\begin{enumerate}[leftmargin=7mm]
\item\label{rmk slope} The above definition can also be formulated in terms of a parametrization satisfying conditions $(i)$ and $(ii)$ in Theorem \ref{theorem:parametrization} (see \cite{PdJ2000}). After applying the $\mathbb C$-algebra automorphism $x\mapsto y$, $y\mapsto x$, if necessary, we may assume that $n\leq m$. In this case, $n$ is the first characteristic exponent of $\gamma$.
Considering higher characteristic exponents, $j\in \mathbb{N}$ is a characteristic exponent if and only if for some $r\in\mathbb N$ we have $x^\frac{m}{n}+x^\frac{m+1}{n}+\ldots+x^\frac{j-1}{n}\in\mathbb C [x^\frac{1}{r}]$ and $x^\frac{m}{n}+x^\frac{m+1}{n}+\ldots+x^\frac{j}{n}\not\in\mathbb C[x^\frac{1}{r}].$ That is, there is an $i$ in Algorithm \ref{alg:Puiseux} such that $\frac{p_i}{q_i}=\frac{j}{n}$ and $K_{(i+1)}$ is a proper extension of $K_{(i)}$.
\item By Theorem \ref{irreducible Puisuex}, all Puisuex expansions of an irreducible Weierstrass polynomial have the same characteristic exponents.
\item \label{rmk char exp splitting}
Suppose $f$ is an irreducible Weierstra\ss polynomial, and $\phi_1,\ldots,\phi_d$ are its Puiseux expansions in terms of $y$, and $G$ is the Galois group of the splitting extentsion $\mathbb{C}(x)\subset \mathbb{C}(x)[\phi_1,\ldots,\phi_d]$.
Then 
the characteristic exponents of $f$ determine the length of the orbit 
under~$G$  of any jet of a Puiseux expansion.

\end{enumerate}
\end{remark}

\begin{defn}
Let $f,g\in \mathbb C\{x,y\}$ be irreducible germs with no common component. Then the \textbf{intersection number} of $f$ and $g$ is
$$i(f,g):=\dim(\mathbb C\{x,y\}/\langle f,g \rangle).$$
\end{defn}

\begin{lemma} \cite[Lemma 5.1.5]{PdJ2000} 
Let $f,g\in\m\subset \mathbb C\{x,y\}$ with $g$ irreducible. If $t\mapsto (x(t),y(t))$ is a parametrization of $g$, then 
\begin{eqnarray*}
i(f,g)&=&\ord_tf((x(t),y(t))\\
&=&\sup\{m\in\mathbb N\mid t^m \text{ divides }f((x(t)),(y(t)))\}.
\end{eqnarray*} 
\end{lemma}

\begin{theorem} \cite[Theorem 5.1.17]{PdJ2000} \label{theorem:faces}
Let $f\in\mathbb C\{x,y\}$ be convenient, let $\Delta_1,\ldots,\Delta_r$ be the facets of $\Gamma(f)$, and $d_i$ be the slope of $\Delta_i$. Then $f=f_1\cdot\ldots\cdot f_r$, where $f_i$ is convenient and $\Gamma(f_i)$ has only one facet of slope $d_i$ for all $i=1,\ldots,r$. 
\end{theorem}

\begin{lemma}\label{lem:PuiseuxIntersectionNewtonBoundary}
Let $f\in\m\subset\mathbb C\{x,y\}$ be a convenient germ with a  non-degenerate Newton boundary. Let $\Delta_1,\ldots,\Delta_r$ be the facets of $\Gamma(f)$, let $d_i$ be the slope of $\Delta_i$, and let $f=f_1\cdot\ldots\cdot f_r$ be the factorization of $f$ as in Theorem \ref{theorem:faces} such that $\Gamma
(f_i)$ is convenient with only one facet of slope $d_i$, which we denote by $\overline{\Delta}_i$. Then the following holds true:

\begin{enumerate}[leftmargin=8mm]
\item If $f_i$ does 
not have a 
weighted linear factor, then $$d_i=-\frac{k_{0}^{ i}}{k_{1}^{ i}} \text{ \hspace{2mm} or \hspace{2mm} } d_i=-\frac{k_{1}^{ i}}{k_{0}^{ i}},$$ where $k_{0}^{ i}$ and $k_{1}^{ i}$ are the first two characteristic exponents of a Puiseux expansion 
of $f_i$. 
\item
With notation and assumptions as in (1) and assuming in addition that $d_j\neq -1$ for some $j\neq i$, the slope $d_i$ is uniquely determined by the characteristic exponents $k_{0}^{i}$, $k_{1}^{i}$, the intersection number of any branch 
of $f_i$ with any branch of $f_j$, and the first characteristic exponent $k_0^j$ of a Puiseux expansion of that branch.
\item\label{lem:PuiseuxIntersectionNewtonBoundary linear} Suppose $f_i$ has a weighted linear factor, and denote this factor by $f_i'$. If for some $j$ we have $d_j\geq d_i\geq -1$ or $-1\geq d_i\geq d_j$, and $f_j$ has a factor $f_j'$ with $f'_i\neq f'_j$,  then $d_i$ is determined by the intersection number of $f_i'$ and $f_j'$, and the first characteristic exponent $k_0^j$ of a Puiseux expansion of $f_j'$.
\item If $d_i\leq -1$, then the sum of the first characteristic exponents associated to the irreducible branches of $f_i$ is the  $x$-degree of the saturation of $f_i$. If $d_i\geq -1$, then the sum yields the $y$-degree of the saturation of $f_i$.
\end{enumerate}
\end{lemma}

\begin{proof}
\begin{enumerate}[leftmargin=8mm]
\item By Theorem \ref{theorem:faces} we have $w(f_i)\dash\jet(f_i,d(f_i))=\sat(\jet(f,\Delta_i))$. Since $f$ has a non-degenerate Newton boundary, $w(f_i)\dash\jet(f_i,d(f_i))$ is non-degenerate. This implies that $w(f_i)\dash\jet(f_i,d(f_i))$ factorizes into $w(f_i)$-homogeneous irreducible polynomials, each factor with multiplicity one. Let $\gamma$ be a Puiseux expansion of any of the factors. Via a right equivalence exchanging the variables,  we may assume that $\gamma$ satisfies $n\leq m$ in Definition~\ref{def:characteristicExponents}.
By Remark \ref{rmk height2}(\ref{lemma:height}), we obtain that $f_i$ is a $y$-general polynomial of $y$-degree $n=k_{0}^{i}$.  Since $f_i$ is not smooth, we have that $n>1$.
Hence, a field extension over $\mathbb{C}$ of degree $>1$ is needed in line \ref{line:fieldExtension} of the first iteration of  Algorithm \ref{alg:Puiseux}, which implies by Remark~\ref{remark:Puiseux}(\ref{rmk slope}) that the slope
of $\overline{\Delta}_i$ is $d_i=-\frac{k_{0}^{i}}{k_{1}^{i}}$. 

\item Suppose $f'_i$ and $f_j'$ are branches of $f_i$ and $f_j$, respectively. Modulo switching the variables, we may assume that $w(y)>w(x)$ for $w=w(\Delta_j)$. Denoting the $y$-degree of $\jet(f_j,\overline{\Delta}_j)$ by~$n$, and using that $\jet(f_j,\overline{\Delta}_j)$ is $y$-general of order $n$, Remark~\ref{rmk height2}(\ref{lemma:height}) implies that $n=k_0^j$.

If $f'_i$ is standard homogeneous, then $d_i=-k_{0}^{i}/k_{1}^{i}=-k_{1}^{i}/k_{0}^{i}$, and there is nothing to prove. Depending on $\tilde w=w(\Delta_i)$ we have the following two cases: If $\tilde w (x)<\tilde w (y)$ then  $d_i=-k_{0}^{i}/k_{1}^{i}$, and a parametrization of $f_i'$ is given by
\begin{equation}\label{firstParametrization}x(t):=t^{k_{1}^{i}}+ \text{ higher order terms}, \hspace{5mm} y(t):=t^{k_{0}^{i}}.\end{equation}
while for $\tilde w (x)>\tilde w(y)$, we have  $d_i=-k_{1}^{i}/k_{0}^{i}$, and a parametrization of $f_i'$ is
\begin{equation}\label{secondParametrization}x(t):=t^{k_{0}^{i}},\hspace{5mm} y(t):=t^{k_{1}^{i}}+\text{ higher order terms,}\end{equation}
 
If (\ref{firstParametrization}) is a parametrization of $f'_i$, then $i(f'_i,f'_j)=n\cdot {k_{0}^{i}}={k_{0}^{j}}\cdot {k_{0}^{i}}$. If (\ref{secondParametrization}) is a parametrization of $f'_i$, then $i(f'_i,f'_j)>n\cdot {k_{0}^{i}} = {k_{0}^{j}}\cdot {k_{0}^{i}}$, which distinguishes between the two cases.

\item 
Without loss of generality, $\jet(f'_i,\Delta_i)$ is of the form $ay+bx^n$. 
Then  $-1\geq d_i\geq d_j$, and $f'_i$ has a parametrization of the form $x(t)=t$, $y(t)=c\cdot t^n+$ higher order terms in $t$, where $c$ is a non-zero constant. Since the $y$-intercept of $f'_j$ is $k_0^j$, we obtain that $i(f'_i,f'_j) = n\cdot  k_0^j$. Here, we use that $d_j\leq d_i$ and $f$ has a non-degenerate Newton boundary, which implies that when plugging the parametrization of $f'_i$ into that of $f_j'$, the lowest non-vanishing order is~$n\cdot  k_0^j$. We hence obtain that $d_i = - k_0^j / i(f'_i,f'_j)$.

\item Any irreducible factor $g$ of $f_i$ is $x$- and  $y$-general.  Hence, by Remark \ref{rmk height2}(\ref{lemma:height}), the first characteristic exponent of the Puiseux expansions of $g$ is either the $x$-degree or the $y$-degree of $f$, depending on whether the slope of the facet is $\leq -1$ or $\geq -1$, respectively.

\end{enumerate}
\end{proof}

\begin{remark}\label{smooth faces}   
Suppose that, in Lemma \ref{lem:PuiseuxIntersectionNewtonBoundary}, the factor $f_i$ does not satisfy any of the assumptions in (1)--(4). Then $f_i$ is a smooth branch,  and the face $\Delta_i$ meets a coordinate axis, without loss of generality the $y$-axis. There are right equivalences which do not change the jets of the faces $\Delta_j$, $j\neq i $ such that the resulting germ has a face $\Delta_i$ with arbitrarily small slope. Moreover, we can achieve  for $\Delta_i$  any $y$-intercept such that $d_i\geq d_j$ for all $j$. \end{remark}

\begin{proof}
We have to consider the case that in Lemma \ref{lem:PuiseuxIntersectionNewtonBoundary}(\ref{lem:PuiseuxIntersectionNewtonBoundary linear}) there does not exist a branch $f_j'$ with the required properties. Then $\jet(f,\Delta_i)$ is of the form $a\cdot x\cdot y^r+b\cdot y^{n+r}$ with constants $a,b\neq 0$. Applying the right equivalence $x\mapsto x-\frac{b}{a}y^n$, $y\mapsto y$, does not change $\jet(f,\Delta_j)$ for $j\neq i$, but strictly increases the $y$-intercept of $\jet(f,\Delta_i)$. By repeating the process, we can achieve that the $y$-intercept is larger than the determinacy of~$f$, which means that the pure $y$-power term and, hence, the face $\Delta_i$ can be removed by a right-equivalence. Then a face $\Delta_i$ with any $y$-intercept $k$ such that $d_i\geq d_j$ can be created by applying a transformation of the form $x\mapsto x+y^{k-r}$, $y\mapsto y$ to the germ.\footnote{Note that $k> r$, since the original face $\Delta_i$ satisfied $d_i\geq d_j$. We may, however, have $k < r+n$ if this is allowed by the convexity condition $d_i\geq d_j$.}
\\
\end{proof}

We collect the observations of the lemma:
\begin{prop}\label{prop branches}
Let $f\in\m\subset\mathbb C\{x,y\}$ be a germ with non-degenerate Newton boundary, and let $f=g_1\cdots g_s$ be a factorization of $f$ into irreducible branches, sorted by increasing slope. Then the intersection numbers and characteristic exponents of the factors $g_i$ determine the Newton polygon of each $g_i$, except that of $g_1$ or $g_s$ in case the facet containing the respective factor is smooth.\footnote{Equivalently, the factor is the only one of its slope and is weighted linear.}
\end{prop}
\begin{proof}
Follows directly from Lemma \ref{lem:PuiseuxIntersectionNewtonBoundary}.
\end{proof}

\begin{corollary}\label{prop:PuiseuxIntersectionNewtonBoundary2}
Suppose that $f\in\m^2\subset\mathbb C\{x,y\}$ is equivalent to a germ with non-degenerate and normalized  Newton polygon $\Gamma$. Then the characteristic exponents and intersection numbers of the branches of $f$ uniquely determine $\Gamma$ (up to permutation of the variables).
\end{corollary}

\begin{proof}
Follows from Proposition \ref{prop branches} since imposing the normalization condition fixes the slope of the Newton polygon of $g_1$ and $g_s$ in case the facet containing the respective factor is smooth.
\end{proof}

\pagebreak[2]
We now turn to the non-degeneracy statement. We first note:

\begin{prop}\label{prop no rational part}
Let $f\in \mathbb C\{x,y\}$ be a germ with non-degenerate Newton boundary. Then no two Puiseux series of $f$ agree up to first non-zero order.
\end{prop}

\begin{proof}
If two Puiseux expansions agree up to first non-zero order, then the jet of the corresponding facet has a multiple factor and, hence, is degenerate. 
\end{proof}

\begin{corollary}\label{cor galois}
Let  $f\in \mathbb C\{x,y\}$ be an irreducible Weierstrass polynomial, let  $\phi_1,\ldots,\phi_d$ be the Puiseux expansions  of 
$f$ in terms of $y$, let $G$ be the Galois group of the splitting extension $\mathbb{C}(x)\subset \mathbb{C}(x)[\phi_1,\ldots,\phi_d]$, and denote by $r$ the lowest order of a non-vanishing term of the expansions. 
Then $f$ has non-degenerate Newton boundary if and only if $\jet(\phi_i,r)$, $i=1,\ldots, d$ form an orbit of length~$d$ under the action of $G$.
\end{corollary}

\begin{proof}
Immediate from Proposition \ref{prop no rational part}.
\end{proof}

\begin{corollary}\label{rmk no rational part}
Let $f\in \mathbb C\{x,y\}$ be a germ with non-degenerate Newton boundary. Then every Puiseux series of $f$ with respect to $y$ is either in $\mathbb C\{x\}$ or does not have a rational part.
\end{corollary}

\begin{proof}
If a branch has $y$-degree $\geq 2$, then it has at least two conjugate Puiseux series, which have the same rational part. By Proposition \ref{prop no rational part}, the rational part has to vanish.
\end{proof}

\begin{lemma}\label{linear degenerate}
Assume that $g\in\m^2\subset\mathbb C\{x,y\}$ has finite Milnor number and maximal Newton number in its right-equivalence class. \begin{enumerate}[leftmargin=8mm]
\item  $\jet(g,\Delta)$ is divisible by $y^2$ in case $w(y)\geq w(x)$, and
\item  $\jet(g,\Delta)$ is divisible by $x^2$ in case $w(x)\geq w(y)$.
  \end{enumerate}

\end{lemma}
\begin{proof}
\begin{enumerate}[leftmargin=8mm]
\item \label{y2}   Suppose $\jet(g,\Delta)$ is not divisible by $y^2$. 
   One then can apply a right-equivalence to the germ to make the jet divisible by $y^2$ by mapping  a linear homogeneous factor of the jet with exponent $\geq 2$ to~$y$. If the resulting Newton polygon is not convenient, one can add the monomial $y^{\mu(g)+2}$ to make the germ convenient. We thus obtain a germ $\tilde{g}$ which is right-equivalent to $g$ and has $N(\tilde{g})>N(g)$, contradicting our assumption on $g$.
   \item Replace $y$ by $x$ in (\ref{y2}).
\end{enumerate}
\end{proof}

The key result is then the following proposition. We provide a proof, which only relies on an analysis of the Puiseux expansions of the germs. For the convenience of the reader we  give all details of the proof.

\begin{prop}\label{prop non-deg}
Suppose that for $f,g\in\m^2\subset\mathbb C\{x,y\}$  there is bijection between the irreducible branches of $f$ and $g$  such that 
\begin{itemize}[leftmargin=5mm]
\item the corresponding characteristic exponents are the same, and
\item the intersection numbers between corresponding branches coincide,
\end{itemize} 
and  suppose that
$f$ is equivalent to a germ with non-degenerate Newton boundary, and $g$ has finite Milnor number. 

Then any germ in the right-equivalence class of $g$ with maximal Newton number has a non-degenerate Newton boundary.
In particular, $g$ is equivalent to a germ  with non-degenerate Newton boundary.

\end{prop}
\begin{proof}
Write $\tilde{f}$ for the germ with non-degenerate Newton boundary which is equivalent to $f$. Assume that $g$ is not equivalent  to a germ with non-degenerate Newton boundary. Choose in the right-equivalence class of $g$ a convenient germ $\tilde{g}$ such that $N(\tilde{g})$ is maximal. \begin{itemize}[leftmargin=5mm]
\item If there is a facet $\Delta$ of $\Gamma(\tilde g)$ with degenerate jet such that the jet has an irreducible factor~$q$ which is not linear homogeneous (and thus all factors, except powers of $x$ and $y$, are not linear homogeneous), then the jet factorizes as $$\jet(\tilde g,\Delta)=q^a\cdot h,$$ with $a\geq 2$, and $q\nmid h$.

Note that, if $\gamma$ is a Puiseux expansion of a branch of $\tilde g$ contributing to $q^a$, and $\phi$ is a Puiseux expansion of the corresponding branch of $\tilde f$, then $\nu(\gamma)=\nu(\phi)$: Since $q$ is non-linear, the orbit length of $\jet(\gamma,\nu(\gamma))$ is at least two, hence, by   Remark \ref{remark:Puiseux}(\ref{rmk char exp splitting}), also the orbit length of $\jet(\phi,\nu(\gamma))$ is at least two, so $\nu(\phi)\leq \nu(\gamma)$. On the other hand, since the branch of $\tilde g$, and hence of $\tilde f$, is non-smooth, Corollary \ref{rmk no rational part} implies that $\phi$ does not have a rational part. So $\jet(\phi,\nu(\phi))$, and hence $\jet(\gamma,\nu(\phi))$, has orbit length at least two, which implies that $\nu(\gamma)\leq \nu(\phi)$, and thus proves our claim.

Note also, that there are exactly $a$ pairwise different branches of $\tilde g$ contributing to $q^a$: Suppose one branch of $y$-degree $a'$ contributes a higher order factor $q^{a_1}$. By Corollary \ref{cor galois}, the $r$-jets of the Puiseux expansions of this branch of $\tilde f$ form an orbit of length $a'$. By Remark~\ref{remark:Puiseux}(\ref{rmk char exp splitting}), the corresponding branch of $\tilde g$ has the same orbit lengths of the jets of the Puiseux expansions. Since this orbit length is equal to $\deg_y(q)$, we have $\deg_y(q) =  a' = a_1 \cdot \deg_y(q)$, a contradiction.

Let $\tilde g_1$ and $\tilde g_2$ be two different branches of $\tilde g$ contributing to $q^a$, and let $\tilde f_1$ and $\tilde f_2$ be the corresponding branches of $\tilde f$. Denoting by $\gamma_2$ and $\phi_2$  Puiseux expansions of $\tilde g_2$ and $\tilde f_2$, respectively, we conclude that
$$i(\tilde g_1, \tilde g_2)>\deg_y(\tilde g_1)\cdot\deg_y(\tilde g_2)\cdot \nu(\gamma_2)=\deg_y(\tilde f_1)\cdot\deg_y(\tilde f_2)\cdot \nu(\phi_2)=i(\tilde f_1,\tilde f_2),$$ a contradiction.%
\footnote{The inequality $i(\tilde g_1, \tilde g_2)>\deg_y(\tilde g_1)\cdot\deg_y(\tilde g_2)\cdot \nu(\gamma_2)$ follows from the fact that the expansions of $\tilde g_1$ and $\tilde g_2$ have the same vanishing order and same lowest order term. Note that the role of $\tilde g_1$ and $\tilde g_2$ can be exchanged in the formula.
The same is true for the formula $i(\tilde f_1,\tilde f_2)=\deg_y(\tilde f_1)\cdot\deg_y(\tilde f_2)\cdot \nu(\phi_2)$, where equality holds since the expansions of $\tilde f_1$ and $\tilde f_2$ have same order, but different lowest order terms.}

\item Now suppose that there is a facet $\Delta$ of $\Gamma(\tilde g)$ with degenerate jet  (after saturation) which factorizes completely into linear homogeneous factors. Write $w=w(\Delta)$. Without loss of generality, we may assume that $w(y)\geq w(x)$. 
By Lemma \ref{linear degenerate}, $\jet(\tilde g,\Delta)$ is then divisible by $y^2$, so it factorizes as$$\jet(\tilde g,\Delta)=l^a\cdot y^b\cdot h,$$ with $l\neq x,y$ linear homogeneous, $a,b\geq 2$, and $l,y\nmid h$. Consider the  factorization of $\tilde g$ according to Theorem \ref{theorem:faces}. Using that $\jet(\tilde g,\Delta)$ is degenerate after saturation, it follows that $\tilde g$ has at least one branch $\tilde g_1$ which divides the $\Delta$-factor of $\tilde g$ and contributes to 
$l^a$, and another branch $\tilde g_2$ which divides a factor corresponding to a facet of the Newton polygon with larger slope than $\Delta$ and contributes to $y^b$.  Let  $\tilde f_1$ and $\tilde f_2$ be the corresponding branches of $\tilde f$, according to the given bijection. 
 
 Note that the Puiseux expansions of $\tilde g_2$ have a higher vanishing order than that of $\tilde g_1$. This implies that $i(\tilde g_1,\tilde g_2)=\deg_y(\tilde g_1)\cdot\deg_y(\tilde g_2)\cdot \nu (\gamma_1)$ where $\gamma_1$ is a Puiseux expansion of $\tilde g_1$. Note also that the Puiseux expansions of $\tilde g_1$ have a non-vanishing rational part.  Since $\tilde g_i$ and $\tilde f_i$ have the same characteristic exponents, $\tilde g_i$ is smooth if and only if $\tilde f_i$ is smooth ($i=1,2$). By Remark~\ref{rmk height2}(\ref{lemma:height}) and the correspondence of characteristic exponents, it follows that any branch dividing the $\Delta$-factor of $\tilde g$ has the same $y$-degree as the corresponding branch of $\tilde f$.
 
Without loss of generality, we may assume that $\tilde g_1$ is chosen as follows: If, among the branches contributing to the factor $l^a$, there is a  non-smooth branch, we choose $\tilde g_1$ as non-smooth. If all the contributing branches are smooth, we choose $\tilde g_1$ such that the Puiseux expansions of $\tilde f_1$ have maximal vanishing  order (among those 
 branches of $\tilde f$ which correspond to branches of $\tilde g$ contributing to $l^a$). We may also assume that $\tilde g_2$ is chosen, 
 among all branches contributing to the factor $y^b$, such that it is non-smooth if a non-smooth branch exists, and otherwise such that its Puiseux expansions have minimal vanishing order. 
 
If $\phi_1$ is a Puiseux expansion of $\tilde f_1$, then $$\nu(\phi_1)>\nu(\gamma_1),$$ as we observe by considering the following two cases:
If $\tilde g_1$ is not smooth, then Corollary~\ref{rmk no rational part} implies that the rational part of the Puiseux expansions of $\tilde f_1$ vanishes.  Since $\tilde g_1$ and $\tilde f_1$ have the same characteristic exponents and the rational part of $ \tilde g_1$ is non-zero, we conclude that $ \nu(\phi_1)>\nu(\gamma_1)$. 
 If $\tilde g_1$ is a smooth branch, then there is at least one other smooth branch $\tilde g_1'$ contributing to~$l^a$. Denote the corresponding branch of $\tilde f$ by $\tilde f_1'$. Let $\gamma'_1$ and $\phi'_1$ be Puiseux expansions of $\tilde g_1'$ and $\tilde f_1'$, respectively. So $\nu(\gamma_1)=\nu(\gamma'_1)$, and, by choice of $\tilde g_1$, 
 $\nu(\phi_1) \geq \nu(\phi'_1)$. If we would have $\nu(\phi_1) \leq \nu(\gamma_1)$, then $$i(\tilde g_1,\tilde g_1')>\nu(\gamma_1) \geq \nu(\phi_1)\geq\nu(\phi'_1) = i(\tilde f_1,\tilde f_1'),$$ using Proposition \ref{prop no rational part} in the  last equality. This contradicts our assumption. Hence, again, we can conclude that  $ \nu(\phi_1)>\nu(\gamma_1)$.

Let now $\gamma_2$ and $\phi_2$ be Puiseux expansions of $\tilde g_2$ and $\tilde f_2$, respectively. 
We observe that $$\nu(\phi_2)\geq \nu(\gamma_2) $$  by considering the following two cases: If $\tilde g_2$ is not smooth,  then (just as above) Corollary~\ref{rmk no rational part} implies that the rational part of the Puiseux expansions of $\tilde f_2$ vanishes. 
 Since $\tilde g_2$ and $\tilde f_2$ have the same characteristic exponents,  
we conclude that $ \nu(\phi_2)\geq \nu(\gamma_2)$.  If $\tilde g_2$ is a smooth branch, then there is at least one other smooth branch $\tilde g_2'$ contributing to~$y^b$. Denote the corresponding branch of $\tilde f$ by $\tilde f_2'$. Let $\gamma'_2$ and $\phi'_2$ be Puiseux expansions of $\tilde g_2'$ and $\tilde f_2'$, respectively. 
If $\nu(\phi_2)<\nu(\gamma_2)$, then 
$$i(\tilde g_2,\tilde g'_2)\geq \nu(\gamma_2)>\nu(\phi_2)\geq i(\tilde f_2,\tilde f'_2),$$
using  Proposition~\ref{prop no rational part} for the last inequality. We thus obtain a contradiction.

If $\nu(\phi_2)>\nu(\phi_1)$, then  $i(\tilde f_1,\tilde f_2)=\deg_y(\tilde f_1)\cdot\deg_y(\tilde f_2)\cdot \nu (\phi_1)$. The same is true in case 
$\nu(\phi_2)=\nu(\phi_1)$, since, by Proposition~\ref{prop no rational part}, for a germ with non-degenerate Newton boundary no two Puiseux expansions agree in lowest order. So we get
$$i(\tilde f_1,\tilde f_2)=\deg_y(\tilde f_1)\cdot\deg_y(\tilde f_2)\cdot \nu (\phi_1) >\deg_y(\tilde g_1)\cdot\deg_y(\tilde g_2)\cdot \nu (\gamma_1) = i(\tilde g_1,\tilde g_2),$$
 which contradicts our assumption, hence, we conclude that $\tilde g$ is non-degenerate. 
 On the other hand, if  $\nu(\phi_1)>\nu(\phi_2)$, then  \begin{align*}i(\tilde f_1,\tilde f_2)&=\deg_y(\tilde f_1)\cdot\deg_y(\tilde f_2)\cdot \nu (\phi_2)\geq \deg_y(\tilde g_1)\cdot\deg_y(\tilde g_2)\cdot \nu (\gamma_2)\\&>\deg_y(\tilde g_1)\cdot\deg_y(\tilde g_2)\cdot \nu (\gamma_1)=i(\tilde g_1, \tilde g_2).\end{align*} 
This again contradicts our assumption, and we conclude that $\tilde g$ is non-degenerate. 
\end{itemize}
\end{proof}

\begin{corollary}\label{prop:PuiseuxIntersectionNewtonBoundary}
Under the same assumptions as in Proposition \ref{prop non-deg},
 suppose that $f$ is equivalent to the germ  $\tilde f$ with a non-degenerate Newton boundary. Then $g$ is equivalent to a germ  with non-degenerate Newton boundary which has the same Newton polygon as $\tilde f$.
\end{corollary}

\begin{proof}
By Proposition \ref{prop non-deg}, $g$ is equivalent to a germ  with non-degenerate Newton boundary. 
Then the claim follows directly from Proposition \ref{prop branches}: If the facet with factor $g_1$ is smooth, then the slope of $\Gamma(g_1)$ can be changed by a right-equivalence (which does not change any terms on any facet of the Newton polygon of $f$ except that with factor $g_1$) to any rational number $-\frac{1}{a}$ which is smaller 
than the slope of $\Gamma(g_2)$, see Remark \ref{smooth faces}. If the facet containing $g_s$ is smooth, we can argue analogously.
\end{proof}

\section{A Classification Algorithm for Isolated Singularities with  a non-degenerate Newton Boundary 
}
\label{section:generalAlg}

On the algorithmic side, our aim is to find for a given input polynomial in $\mathbb Q[x,y]$, that  over $\mathbb{C}$ is right-equivalent to a germ with a non-degenerate Newton boundary, an explicit right-equivalence to a normal form equation as in Theorem \ref{theorem:normalform}. The first step in the process is to transform the input polynomial to a germ with a non-degenerate Newton boundary. In the present section, we describe an algorithm for this purpose. Our algorithm is based on the following observations (the first two being obvious and the third implied by Remark \ref{rmk sect 4 1}):

\begin{remark} \label{remark:degenerate}
A (weighted) homogeneous polynomial $f\in \mathbb{C}[x,y]$ is degenerate (that is, its Milnor number is infinite, equivalently, it has a non-isolated singularity at zero) if and only if one of its irreducible factors has multiplicity larger than one. 
\end{remark}

\begin{remark} \label{remark:degenerate2}
A germ $f$ has a non-degenerate Newton boundary 
if the Milnor number of the saturation of $\jet(f,\Delta)$ for each  facet $\Delta$ of the Newton polygon of $f$ is finite. 
\end{remark}

\begin{remark} \label{remark:degenerate3}
Suppose $f$ is in the right-equivalence class of a germ with non-degenerate Newton boundary. Then $f$ has a non-degenerate Newton boundary if and only if it has maximal Newton number (within the right-equivalence class).
\end{remark}

The basic idea of the classification algorithm in this section is now as follows:  Let $g$ be a weighted homogeneous  polynomial (think of the jet of the input polynomial with regard to a facet of its Newton polygon). If $g$ has no factor of multiplicity larger than $1$, then $g$ is non-degenerate. The number, degree or multiplicity of factors of $g$ is invariant under automorphisms. Hence, if $g$ has more than two different linear homogeneous factors of multiplicity larger than $1$ or if it has a non-linear homogeneous factor of multiplicity larger than $1$, it cannot be mapped by a right-equivalence to a polynomial with a non-degenerate saturation. 

We first consider the case that $g$ is homogeneous with respect to the standard grading. If $g$ has exactly one linear homogeneous factor with multiplicity bigger than $1$, we can map this factor to $x$. If $g$ has exactly two linear homogeneous factors with multiplicity bigger than $1$, we can map these factors to $x$ and $y$, respectively. In both cases, after the linear transformation, the polynomial is non-degenerate after saturation.  We thus obtain Algorithm \ref{alg:standhomclas}.

\begin{algorithm}[h]
\caption{Transform the standard homogeneous $d$-jet of a polynomial, where $d$ is the maximal filtration of the polynomial, to a polynomial with a non-degenerate saturation.}%
\label{alg:standhomclas}
\begin{spacing}{1.05}
\begin{algorithmic}[1]

\Require{A polynomial $f\in\m^3, f\in \mathbb Q[x,y]$ with maximal filtration $d$.}

\Ensure{A polynomial $h\in \mathbb K[x,y]$ defined over an extension field $\mathbb Q \subset \mathbb K \subset \overline{\mathbb K}$ such that $h$ is right-equivalent to $f$, and  $\sat(\jet(f,d))$ is a non-degenerate polynomial (possibly term), if such an $h$ exists, and false otherwise.
}
\State $g:=\jet(f,d)$.
\State Factorize $g=c\cdot g_1^{l_1}\cdot\ldots\cdot g_n^{l_n}$, where $l_1\ge l_2\ge\cdots\ge l_n$, $c\in\mathbb Q$, $g_1\in \mathbb K[x,y]$ with $\mathbb K =\mathbb Q$ in case $n=1$, and $g_1,\ldots,g_{n}\in\overline{\mathbb K}[x,y]$ linear homogeneous and coprime, $g_1,g_2\in \mathbb K[x,y]$ with $[\mathbb K:\mathbb Q]$ minimal (among all admissible choices of $g_1,g_2$), in case $n\geq 2$.
\If{$n=1$}
\State Apply  a linear automorphism to $f$ sending $g_1\mapsto x$. \label{line:trans1}
\Else
\State Apply $g_1\mapsto x$, $g_2\mapsto y$ to $f$.\label{line:trans2}
\EndIf
\If{the saturation of $\jet(f,d)$ is degenerate}
\Return{\texttt{false}}
\EndIf
\Return{$f$}
\end{algorithmic}
\end{spacing}

\end{algorithm}

\begin{proof}
Correctness of the algorithm follows directly from Remark \ref{remark:degenerate} and the above discussion.
\end{proof}

\begin{remark}
To match the actual implementation, the input polynomial in Algorithm \ref{alg:standhomclas} is assumed to be defined over $\mathbb Q$. Note that the transformations that are applied in lines~\ref{line:trans1} and~\ref{line:trans2} of Algorithm \ref{alg:standhomclas} may require to pass to an algebraic extension of the base field. 
\end{remark}

\begin{algorithm}[ht]
\caption{Transform a non-standard weighted homogeneous $d$-jet of a polynomial, where $d$ is the maximum weighted filtration of the polynomial, to a polynomial  with a non-degenerate saturation.}%
\label{alg:homclas}
\begin{spacing}{1.05}
\begin{algorithmic}[1]

\Require{A polynomial $f\in\m^3$, $f\in\mathbb K[x,y]$, where $\mathbb K$ is an extension field of $\mathbb Q$, and a weight $w=(w_1,w_2)$ with $w_1\neq w_2$ the weight of one of the facets of the Newton boundary of $f$.}

\Ensure{A polynomial $h\in\mathbb L[x,y]$ defined over an extension field $\mathbb K \subset \mathbb L\subset \overline{\mathbb K}$, such that $h$ is right-equivalent to $f$, where  $\jet(h,w)$ is a polynomial (possibly term) with a non-degenerate saturation, if such an $h$ exists, and false otherwise.
}
\If{$w_1<w_2$}
\State $g:=\sat(\jet(f,w),x)$.
\EndIf
\If{$w_1>w_2$}
\State $g:=\sat(\jet(f,w),y)$.
\EndIf
\State Factorize $g=cg_1^{l_1}\cdots g_n^{l_n}\widetilde{g}$, where $c\in\mathbb K$, $l_1\ge l_2\ge\cdots\ge l_n$ and  $g_1\ldots,g_n\in \overline{\mathbb K}[x,y]$ weighted linear homogeneous and coprime with $n$ maximal, $\widetilde{g}\in \mathbb K[x,y]$ a product of non-associated irreducible polynomials in $\overline{\mathbb K}[x,y]$, and $g_1\in\mathbb L[x,y]$ with $[\mathbb L:\mathbb K]$ minimal.
\If{$l_2>1$}
\Return{\texttt{false}}
\EndIf
\If{$\widetilde{g}$ is degenerate}
\Return{\texttt{false}}
\EndIf
\If{$w_1<w_2$ \textbf{and} $n\geq 1$}\label{line:if1}
\State Apply right equivalence to $f$ which sends $g_1\mapsto y, x\mapsto x$. \label{line:trans21}
\EndIf
\If{$w_1>w_2$ \textbf{and} $n\geq 1$}\label{line:if2}
\State Apply right equivalence to $f$ which sends $g_1\mapsto x, y\mapsto y$. \label{line:trans22}
\EndIf
\Return{$f$}

\end{algorithmic}
\end{spacing}

\end{algorithm}

A similar approach can be used to decide whether a non-standard weighted homogeneous polynomial can be transformed by a weighted linear homogeneous transformation to a polynomial with non-degenerate saturation. Again, if $g$ has exactly one linear homogeneous factor with multiplicity bigger than $1$, we can map this factor to $x$ or $y$, in case $w_1>w_2$ or $w_2>w_1$, respectively. After this weighted linear transformation, the polynomial is non-degenerate after saturation. If $g$ has two (or more) linear homogeneous factors with multiplicity bigger than one,  it cannot be mapped by a right-equivalence to a polynomial with a non-degenerate saturation. The approach is summarized in Algorithm~\ref{alg:homclas}.

\begin{proof}
Correctness of the algorithm follows directly from Remark \ref{remark:degenerate} and the above discussion. 
\end{proof}

\begin{remark}\label{weighted field extension}
In Algorithm \ref{alg:homclas}, 
the transformations in line \ref{line:trans21} to \ref{line:trans22} are defined over the ground field~$\mathbb K$, except if all homogeneous linear factors appear with exponent one: Consider a factorization $$g=c\cdot g_1^{l_1}\cdot\ldots\cdot g_n^{l_n}\cdot \widetilde{g},$$ where $c\in\mathbb K$, $l_1\ge l_2\ge\ldots\ge l_n$ and $g_1,\ldots,g_n\in\overline{\mathbb K}[x,y]$ are linear homogeneous and coprime, with $n$ maximal, and $\widetilde{g}\in\overline{\mathbb K}[x,y]$. If $l_1>1$ and $g_1$ has a conjugate, then $g$ cannot be transformed to a germ with non-degenerate Newton boundary. Otherwise $g_1$ can be chosen with coefficients in $\mathbb{K}$, which implies that the corresponding weighted linear transformation in the algorithm is defined over $\mathbb K$.
\end{remark}

\begin{remark}
Note that the if-clauses in line \ref{line:if1} and \ref{line:if2} of Algorithm \ref{alg:homclas} ensure that the transformations in lines \ref{line:trans21} and \ref{line:trans22} are of $w$-degree $0$, and hence do not create terms below the Newton polygon.
\end{remark}

Using Algorithms \ref{alg:standhomclas} and \ref{alg:homclas}, we now state Algorithm \ref{alg:clas}, which, for a germ which is equivalent to a germ with non-degenerate Newton boundary,  determines a right-equivalent germ with normalized non-degenerate Newton boundary.

\begin{remark}
More generally, given a germ $f\in \mathbb{Q}[x_1\ldots,x_n]$ of corank $\leq 2$ with finite Milnor number, we can obtain a germ which has a non-degenerate Newton boundary and is stable equivalent to $f$ by first applying the Splitting Lemma to $f$. In case $\corank(f)=2$, we continue with Algorithm \ref{alg:clas} applied to the degenerate part of $f$. If $\corank(f) = 1$, after applying the splitting lemma, we can just return $\jet(f,d)$, where $d$ is the  maximal filtration of $f$ with respect to the standard grading. Note that then the Newton polygon is given by the vertex $\mu(f)+1$.
If $\corank(f)=0$, that is, $f$ has an $A_1$-singularity, we just return $0$.
\end{remark}

\begin{remark}Using the notation of Remark \ref{weighted field extension}, the case $l_1=1$ can only occur in the last call of Algorithm \ref{alg:homclas} from Algorithm \ref{alg:clas} with regard to facets with $w_2>w_1$, and similarly in the last call with regard to facets with $w_1>w_2$. Hence, the ordering in which Algorithm \ref{alg:clas} considers the facets of the Newton polygon can be chosen in a way that these (one or two) calls, and hence towers of field extensions, occur only in the last steps of the algorithm. This improves the performance and simplifies the implementation. 
\end{remark}

\begin{algorithm}[p]
\caption{Transform a germ to a germ with a non-degenerate, convenient Newton boundary.}%
\label{alg:clas}
\begin{spacing}{1.05}
\begin{algorithmic}[1]

\Require{A polynomial germ $f\in\mathbb Q[x,y]$, $f\in\m^3$, of corank $2$ with finite Milnor number.}

\Ensure{A polynomial $g$ with a non-degenerate Newton boundary which is right-equivalent to $f$, if such $g$ exists, and \texttt{false} otherwise. Moreover, $g$ is convenient and satisfies the normalization condition in Definition~\ref{def normal cond}.
}
\State $d:=$  maximal filtration of $f$ w.r.t. the standard grading.
\If{$\mu(\sat(\jet(f,d)))=\infty$ or the normalization condition is not satisfied}
\State apply Algorithm \ref{alg:standhomclas} to $f$.
\EndIf
\If{$f=\texttt{false}$}
\Return{\texttt{false} (the germ has a degenerate Newton boundary)}\label{line:terminate}
\EndIf
\State $S_0:=$ the set of monomials of $\jet(f,d)$ that lie on the vertices of $\Gamma(f)$.\label{line:S0}
\State $S_1:=\emptyset$ \label{line:S1}
\While{{\bf true}} \label{while:big1}
\If{ $\Gamma(f)$ does not intersect the $x$-axis} \label{line:intersect1}
$f=f+x^a$, where $a=\mu(f)+2$.\label{line:intersect2}
\EndIf
\If{ $\Gamma(f)$ does not intersect the $y$-axis } \label{line:intersect1a}
$f=f+y^a$, where $a:=\mu(f)+2$.\label{line:intersect2a}
\EndIf
\State Let $\Delta_1,\Delta_2,\ldots,\Delta_v$ be the facets of $\Gamma(f)$ ordered by increasing slope.
\If{$\mu(\sat(\jet(f,\Delta_i)))<\infty\,\forall i=1,\ldots,v$ and the normalization condition is satisfied}
\Return{($f$,$\Gamma(f)$)}
\Else
\State Let $m$ be an element of $S_0$.
\State Let $\delta_1,\ldots,\delta_r$ ($r\leq 2$) be the facets of $\Gamma(f)$, ordered by increasing slope, adjacent to~$m$.
\For{$i$ \textbf{from} $1$ \textbf{to} $r$}
\State $h:=\jet(f,\delta_i)$, $w:=$ the weight defined by $h$.
\While{$\mu(\sat(h))=\infty$ or the normalization condition is not satisfied} \label{line:beginWhileFace} 
\State $f:=$ output of Algorithm \ref{alg:homclas} applied to $f$ and $w$.
\If{$f=\texttt{false}$}
\Return{\texttt{false} (the germ has a degenerate Newton boundary)}\label{line:terminate}
\EndIf
\State Let $\delta_1,\ldots,\delta_r$ 
be the facets of $\Gamma(f)$, ordered by increasing slope, adjacent to $m$.
\State $h:=\jet(f,\delta_i)$, $w:=$ the weight defined by $h$.
\EndWhile\label{line:endWhileFace}
\EndFor
\State Let $\eta_1,\ldots,\eta_r$ ($r\leq 2$)  be the facets of $\Gamma(f)$, ordered by increasing slope, containing~$m$.
\State $S_1:=S_1\cup\{m\}$ \label{line:addm}
\State $f_1:=\jet(f, \eta_1\cap \ldots \cap\eta_r))$.
\State $S_0:=$ monomials in $(S_0\cup\supp(f_1))\setminus S_1$ lying on the vertices of $\Gamma(f)$. \label{line:changeS0}
\EndIf
\EndWhile\label{while:big2}
\If{$\sat(\jet(f,\Delta_1))$ is smooth}
\For{$d\le l\le\mu+2$}\label{for:satbegin}
\If{$\Delta_1$ cuts the $x$-axis at $x^l$}
\State Write $\jet(f,\Delta_1)$ as $\jet(f,\Delta_1)=c_1x^l+c_2x^sy$
\State Apply $y\mapsto y-\frac{c_1}{c_2}x^{l-s}$, $x\mapsto x$
\State Recalculate $\Delta_1$
\EndIf
\EndFor
\EndIf
\If{$\Delta_1$ does not cut the $x$-axis at $x^l$, where $l=\mu+2$}
$f = f+x^l$
\EndIf
\If{$\sat(\jet(f,\Delta_v))$ is smooth}
\For{$d\le l\le\mu+2$}
\If{$\Delta_v$ cuts the $y$-axis at $x^l$}
\State Write $\jet(f,\Delta_v)$ as $\jet(f,\Delta_v)=c_1y^l+c_2y^sx$
\State Apply $y\mapsto x-\frac{c_1}{c_2}y^{l-s}$, $y\mapsto y$
\State Recalculate $\Delta_v$
\EndIf
\EndFor\label{for:satend}
\EndIf
\If{$\Delta_v$ does not cut the $y$-axis at $y^l$, where $l=\mu+2$}
$f = f+y^l$
\EndIf
\Return $f$

\end{algorithmic}
\end{spacing}

\end{algorithm}

\begin{remark}
Composition of all right-equivalences applied in the course of Algorithm \ref{alg:clas} yields an explicit polynomial right-equivalence from the input germ to a germ with normalized non-degenerate Newton boundary. Addition of the monomials $x^a$ and $y^a$ in lines \ref{line:intersect1} and \ref{line:intersect1a}, can also be realized by a polynomial right-equivalence (creating higher order terms). Note that, after each right-equivalence in the algorithm, we may truncate above the determinacy. These truncations amount to non-polynomial right-equivalences and thus are not explicitly known. However, a lower bound (depended on the truncation order) for the filtration of the truncation transformations can be specified.
\end{remark}

\paragraph{\textit{Proof of Algorithm \ref{alg:clas}}.}
We show that the algorithm will either transform $f$ in finitely many steps to a germ with a non-degenerate Newton boundary  which is normalized according to the conditions in Definition~\ref{def normal cond}, or, if the germ is not right-equivalent to a germ with a non-degenerate Newton boundary, will return false. We write $T$ for the type associated to this Newton boundary, which we will denote by $\Gamma(T)$.

It is sufficient to prove that if $f$ can be transformed to a germ with these properties, the algorithm finds such a transformation in finitely many steps. So assume that there is a right-equivalence $\psi$ such that $\tilde f  = \psi(f)$ satisfies the above properties. We will iteratively change $\psi$ by composing it with suitable right-equivalences until $\psi_0^{w(T)_i}=\operatorname{id}$ for all $i$.  We will do this by considering jets of $f$ with regard to a suitable sequence $w_j$ of weights. In the $i$-th iteration we find a right-equivalence $\phi$ such that the $w_j\dash\jet(\phi(f),d(f))=w_j\dash\jet(\psi(f),d(f))$ for all $j\leq i$. Note that, before the $i$-th iteration, $w_i$ will be an entry of the current $w(f)$, and any entry of $w(T)$ will occur among the $w_j$'s. We then replace $f$ by $\phi(f)$ and $\psi$ by $\psi\circ \phi^{-1}$ and iterate (denoting in every iteration, for simplicity of notation, the current germ by $f$, and the current transformation to $\tilde f$ by $\psi$). 

We start out by considering $\jet(f,d)$, where $d$ is the maximal filtration of $f$ with respect to the standard grading. Let $$\jet(f,d)=c\cdot g_1^{l_1}\cdot \ldots\cdot g_n^{l_n}$$ be the factorization of this jet as in Algorithm~\ref{alg:standhomclas}, where $l_1\ge l_2\ge\ldots\ge l_n$, $g_1,\ldots,g_{n}$ are linear homogeneous. There are the following cases:

\begin{itemize}[leftmargin=4mm]
\item If more than two of the $l_i$ are larger than $1$, then $f$ cannot be transformed to a germ with non-degenerate Newton boundary.

\item Suppose $l_1\geq l_2>1$, and the remaining $l_i=1$ then, up to permutation and scaling of the variables, the only way to obtain a non-degenerate jet is to map $g_1\mapsto x$ and $g_2\mapsto y$.

\item Suppose that $l_1=\ldots=l_n=1$. Note that this is the case of an ordinary $n$-fold point. If $n=1$ then $f$ is a smooth germ. If $n\geq 2$, the lowest jet $\jet(f,d)$ is already non-degenerate. Algorithm~\ref{alg:clas} then maps $g_i\mapsto x$ and $g_j\mapsto y$ for some choice of $i\neq j$, adds terms to create smooth facets towards the coordinate axes to satisfy our normalization condition, if necessary, and then finishes without any further iteration of the main loop. Hence our choice of $i$ and $j$, and of the normalization, will not influence subsequent iterations of the algorithm. The choice of $i$ and $j$ also does not influence the face of the Newton polygon corresponding to $\jet(f,d)$ after the transformation (and any further facets have a smooth saturation and cut one of the coordinate axis, so are unique according to our normalization condition). Note that,  after the transformation, $\jet(f,d)$ is a term if and only if $n=2$.

\item Lastly, suppose that $l_1>1$ and $l_2=\ldots=l_n=1$. To obtain a non-degenerate lowest jet, we have to map $g_1$ to a non-zero multiple of a variable. For what follows, we fix the choice  $g_1\mapsto x$. This determines one vertex of the face corresponding to $\jet(f,d)$ and makes the saturation of the face non-degenerate (but not necessarily normalized). 
If $n=1$, the algorithm maps $y\mapsto y$.
 If $n\geq2$, in order to satisfy our normalization condition, the algorithm chooses an $i\neq 1$, maps $g_i \mapsto y$ and, if necessary, adds a smooth facet towards the coordinate axis.
 
We will show that if for one of these choices of transformation $\rho_0:g_1\mapsto x, g_i\mapsto y$, the algorithm determines in subsequent steps that there is no transformation to a germ with a non-degenerate Newton boundary, then this is also the case for all other choices. 

 For this we prove in the case $n\geq 2$, that if $g$ is \emph{any} 
linear divisor of $\jet(f,d)$ with $g\neq g_1=x$, then there is a right-equivalence $\rho$ with $\rho_0(g)=y$ such that $\rho(f)$ is a germ with non-degenerate Newton boundary: 
By assumption, for $g_0:=(\psi^{-1})_0(y)$, we have $\psi_0(g_0)= y$ and $\psi(f)$ has a non-degenerate and normalized Newton boundary. Let $w(\psi(f))=(w_1,\ldots,w_s)$ be the associated piecewise weight, and let $\Delta_1,\ldots,\Delta_s$ be the corresponding facets of $\Gamma(\psi(f))$. Note that $\jet(\psi(f),\Delta_s)$ has smooth saturation since $\jet(\psi(f),d)$ has a factor $y$, but none of the other factors is divisible by $y$, that is, $\jet(\psi(f),d)$ has a monomial of the form $yx^l$, $l>1$. If $g$ is any linear divisor of $\jet(f,d)$ with $g\neq x$, and $\lambda$ is the linear transformation $\lambda=(x\mapsto x,\, g \mapsto y)\circ(x\mapsto x,\, y\mapsto g_0)$, then $\rho=\lambda\circ\psi$ maps $g$ to $y$, and $\jet(\rho(f),\Delta_i)$ and $\jet(\psi(f),\Delta_i)$ agree for $i=1,\ldots,s-2$ up to scaling of $y$. To see this, note that, since for all the weights $w_1,\ldots,w_s$ we have $w_i\dash\deg(y)\ge w_i\dash\deg(x)$, a linear transformation will only create terms of $(w_1,\ldots,w_{s-1})$-weight bigger than or equal to $d(\psi(f))$. Moreover, since we choose $g$ to be a linear factor of $\jet(f,d)=\jet(f,\Delta_{s-1})$, the support of the $d$-jet does not change under $\lambda$. The facet $\Delta_s$ with smooth saturation may still change. If this facet is eliminated, we modify $\rho$ by a right-equivalence of higher filtration to make the Newton polygon convenient again, thus satisfying the normalization condition. 
In the case $n=1$, we have to prove that if $g$ is \emph{any} linear form which is transversal to $l_1=x$ then there is a right-equivalence $\rho$ with $\rho_0(g)=y$ such that $\rho(f)$ is a germ with normalized, non-degenerate Newton boundary. For this we can use a similar argument as above with the only difference that there is no smooth facet towards the $x$-axis, and $\jet(\rho(f),\Delta_i)$ and $\jet(\psi(f),\Delta_i)$ agree for $i=1,\ldots,s$ up to scaling of $y$.

\end{itemize}

In any of the above cases (except of course the first one, where the algorithm finishes with \texttt{false}), we can now assume without loss of generality that the algorithm has chosen the linear transformation as $\psi_0^{-1}$, that is, we may assume that $\psi_0$ is the identity. In particular, we can assume that $\jet(f,d)$ has non-degenerate saturation and (if it is not just a term) satisfies the normalization condition. 

If $f$ has a normalized non-degenerate Newton boundary we are finished. If the Newton boundary is non-degenerate and all facets except the smooth facets are normalized, we can normalize  these facets and are also finished. So suppose this is not the case. Starting with the face corresponding to the lowest standard graded jet, we find $\Gamma(T)$ by inductively determining a facet of the Newton polygon adjacent to the faces already coinciding with those of $\Gamma(T)$. In every iteration, if $f$ does not have terms on both coordinate axes, we add terms as described in lines \ref{line:intersect1} and \ref{line:intersect1a} of Algorithm \ref{alg:clas} such that $\Gamma(f)$ does not have infinite faces. In the following, suppose $m$ is a vertex monomial of $\Gamma(T)$ and $\delta$ is a facet of $\Gamma(f)$ adjacent to $m$ which has not been considered so far in the sense of verifying whether $\delta$ is a facet of $\Gamma(T)$.\footnote{ For book-keeping purposes in this iteration the algorithm uses two sets of monomials $S_0$ and $S_1$. The set $S_0$ contains all monomials that correspond to vertices of $\Gamma(T)$ and have at least one adjacent facet that is not, or has not yet been verified to be a facet of $\Gamma(T)$. The set $S_1$ consists of all monomials in $S_0$ for which we have verified that for its adjacent facets $\delta_1,\ldots,\delta_r$ ($r\le 2$), $\jet(f,\operatorname{span}(\delta_1,\ldots,\delta_r))$ has non-degenerate saturation and is normalized (hence, is a facet of $\Gamma(T)$).}
As in Algorithm \ref{alg:homclas}, write $$g:=\jet(f,\delta)=cg_1^{l_1}\cdot\ldots\cdot g_n^{l_n}\widetilde{g}$$where $l_1\ge l_2\ge\ldots\ge l_n$ and $g_1,\ldots,g_n$ are (weighted) linear homogeneous and coprime, and $n$ is maximal. Note that if $n=0$, then $g$ is non-degenerate and satisfies the normalization condition.

The next step is to find a transformation $\phi$ such that $\widetilde w\dash\jet(\phi(f),k)=\widetilde w\dash\jet(\psi(f),k)$, where $\widetilde w=(\widetilde w_1,\widetilde w_2)$ is a weight associated to $\delta$ and $k=\widetilde w\dash\deg(g)$. We have the following cases:

\begin{itemize}[leftmargin=4mm]

\item If more than one of the $l_i$ is bigger than $1$, then (since $\widetilde{w}$ is not the standard weight)  $f$ cannot be transformed to a germ with a non-degenerate Newton boundary. In this case, Algorithm \ref{alg:homclas} returns $\texttt{false}$ (see line \ref{line:terminate}). 

\item If $g$ is non-degenerate and  the normalization condition is satisfied we pass on to the next facet.\end{itemize}

\begin{itemize}[leftmargin=4mm]

\item Suppose now $l_1>l_2=1$. Without loss of generality, we may assume that  $\widetilde w_1<\widetilde w_2$ and $g_1=c_1y+h$, where $c_1$ is a constant and the standard degree of $h$ is bigger than $1$. Algorithm~\ref{alg:homclas} applies the transformation $\phi:g_1\mapsto y$, $x\mapsto x$ to $f$, which maps $g$ to a polynomial with a non-degenerate saturation which, in case it is not a term, satisfies the normalization condition. If $g$ is transformed to a term $m$, a new adjacent facet of $m$ appears with a new associated weight $\widetilde w$ and we repeat the process via the while-loop (line \ref{line:beginWhileFace} to line \ref{line:endWhileFace}). On the other hand, if $g$ is not transformed to a term, the slope of $g$ is preserved, and we pass to the next facet. 

\item Finally, suppose $l_1=l_2=\ldots=l_n=1$. In this case, $g$ is non-degenerate, but does not necessarily satisfy the normalization condition. Without loss of generality we may assume that $\widetilde w_1>\widetilde w_2$. Then Algorithm~\ref{alg:homclas} maps $g_i\mapsto x$ for some choice of $i$ and $y\mapsto y$. Similar to the standard weighted case, we will show that, if for the choice $\rho_0^{\widetilde w}: g_i\mapsto x, y\mapsto y$ the algorithm determines in subsequent steps that there is no transformation to a germ with a non-degenerate Newton boundary, then this is also the case for all other choices of $i$. 
For this we prove that, if $h$ is any $\widetilde w$-weighted linear divisor of $g$, then there is a right-equivalence $\rho$ with $\rho_0^{\widetilde w}(h)=x$ such that $\rho(f)$ is a germ with a non-degenerate Newton boundary: By assumption, for $h:=(\psi^{-1})_0^{\widetilde w}(x)$, we have $\psi_0^{\widetilde w}(h)=x$ and $\psi(f)$ has a non-degenerate Newton boundary. Let $w(\psi(f))=(w_1,\ldots,w_s)$ be the associated piecewise weight, and let $\Delta_1,\ldots,\Delta_s$ be the corresponding facets of $\Gamma(\psi(f))$. Note that $\jet(\psi(f),\Delta_1)$ has a smooth saturation since $\psi(g)$ has a monomial of the form $xy^l$, $l>1$. We can now use a similar argument as in the standard weighted case to prove that $\rho$ transforms $f$ to a germ with a non-degenerate normalized Newton boundary.  

\end{itemize}

Note that a facet with a non-degenerate saturation which satisfies the normalization condition cannot be transformed  by a weighted linear transformation with regard to the weight of the facet to a term or a different facet. Hence, in Algorithm \ref{alg:clas}, the vertex points of $\jet(f,\operatorname{span}(\delta_1,\ldots ,\delta_r))$ lie on $\Gamma(T)$ after each iteration of the while-loop (line  \ref{while:big1} to line  \ref{while:big2}). Lines \ref{for:satbegin} to \ref{for:satend} ensure that the facets which have smooth saturation and cut the coordinate axis satisfy the normalization condition.

We now prove that the algorithm terminates after finitely many iterations. For that we have to show that the while-loop in lines \ref{while:big1} to \ref{while:big2} and the while loop in lines \ref{line:beginWhileFace} to \ref{line:endWhileFace}  terminates. 
After a finite number of iterations of the loop in lines \ref{line:beginWhileFace} to \ref{line:endWhileFace}, the loop either terminates or the highest $x$-power or the highest $y$-power below $\Gamma(f)$ increases.  If $x^{m_1}$ and $y^{m_2}$ are the highest powers of $x$ and $y$ below $\Gamma(f)$, then $1,x,...,x^{m_1-1},y,...,y^{m_2-1}$ represent different equivalence classes in the local algebra. Hence, since $\mu(f)$ is finite, the process must stop after finitely many iterations. In each iteration of the while-loop in lines \ref{while:big1} to \ref{while:big2} the process may either stop (in case $f$ is not equivalent to a germ with non-degenerate Newton boundary), or create at least one new non-degenerate facet of $\Gamma(f)$. Since the number of facets of a Newton polygon with a fixed minimum degree monomial is bounded, the algorithm terminates.
\begin{flushright}$\square$\end{flushright}

Note that from the proof of Algorithm \ref{alg:clas} we in particular obtain uniqueness of the Newton polygon determined by the algorithm. This follows also from Corollary \ref{prop:PuiseuxIntersectionNewtonBoundary2} using the invariance of the charateristic exponents and intersection numbers under right-equivalence.

\begin{corollary}
Let $f\in\mathbb C[x,y]$ be a germ which is right-equivalent to a germ with non-degenerate Newton boundary $\Gamma$. If $\Gamma$ is assumed to satisfy the normalization condition, then $\Gamma$ is uniquely determined by $f$ (up to permutation of the variables).
\end{corollary}

The above algorithms are implemented in the \textsc{Singular} library \texttt{arnold.lib}. We illustrate the implementation at the following example:
\begin{example}\label{ex nb}
We start out with a germ with non-degenerate Newton boundary and apply the automorphism $\phi$. The result is a rather generic polynomial and has a degenerate Newton boundary. We then use our algorithm to transform this germ to a germ with
non-degenerate Newton boundary. Note that our implementation is based (both for input and output) on the \Singular\  data type \texttt{Poly} which (in contrast to \texttt{poly}) allows us to handle polynomials that are not in the current active ring. In particular, this allows us to handle extensions of the coefficient field in a transparent manner.
\begin{verbatim}
> LIB "Arnold.lib";
> ring R = 0,(x,y),ds;  
> Poly g = 4*x^2*y^4+x^3*y^3+y^2*x^4+5*x^10+y^13;
> Poly phix = x+y^2+x^2+x*y+x^2*y+x*y^3;
> Poly phiy = y+y^2+2*x^2+x*y+y*x^2+y^2*x+x*y^4;
> RingHom phi;
> phi.source = R;
> phi.target = R;
> phi.images = list(phix,phiy);
> Poly f = phi(g);
> newtonPolygon(f);
Polygon with
Vertices: [0,6], [6,0]
Facets:   [[0,6], [6,0]]
Normals:  [1,1]
> def (F, P) = classifyNewtonBoundary(G);
\end{verbatim}

\noindent This function call returns a germ $F$ and a Newton polygon $P$, where $F$ which is right-equivalent to $f$ and has a non-degenerate normalized Newton boundary $P$. Here, $P$ is  represented by the data type \texttt{Polygon}, which stores the vertices, one dimensional facets and their normal vectors:

\begin{small}
\begin{verbatim}
> F;
 4*x4y2+x3y3+x2y4+24*x5y2+30*x4y3+12*x3y4+6*x2y5+28*x6y2+86*x5y3+25*x4y4-4*x3y5-5*x2y6
-4*xy7-40*x7y2-104*x6y3-151*x5y4-223*x4y5-159*x3y6-78*x2y7-16*xy8-32*x9y-48*x8y2
-263*x7y3-277*x6y4+397*x5y5+396*x4y6+90*x3y7+18*x2y8+12*xy9+9*y10-128*x10y-312*x9y2
+226*x8y3+544*x7y4+2016*x6y5+3202*x5y6+554*x4y7-366*x3y8-6*x2y9+102*xy10+58*y11
+24*x11y-387*x10y2-195*x9y3+2150*x8y4-10*x7y5-2795*x6y6-5869*x5y7-8957*x4y8-3989*x3y9
-462*x2y10+544*xy11+129*y12+x13+248*x12y+1853*x11y2+2163*x10y3+3042*x9y4+10302*x8y5
-2426*x7y6-6092*x6y7+7112*x5y8+8012*x4y9+8600*x3y10+6128*x2y11+2834*xy12-368*y13
+77*x14-131*x13y-713*x12y2+833*x11y3-17096*x10y4-28769*x9y5-2038*x8y6-47695*x7y7
-14989*x6y8+33021*x5y9-5681*x4y10-6632*x3y11+7616*x2y12-164*xy13-3138*y14+193*x15
+1565*x14y+429*x13y2-7032*x12y3-98*x11y4-49796*x10y5+7939*x9y6+130879*x8y7
-59808*x7y8+48755*x6y9+39932*x5y10-99485*x4y11-13597*x3y12+28520*x2y13-21596*xy14
-5924*y15
> P;
Polygon with
vertices: [0,10], [2,4], [4,2], [13,0]
facets:   [[0,10], [2,4]], [[2,4], [4,2]], [[4,2], [13,0]]
normals:  [3,1], [1,1], [2,9]
\end{verbatim}
\end{small}
\end{example}

\section{Computing the Modality and a Regular Basis for a Germ with a non-degenerate Newton Boundary
}
\label{section:modality}

In \cite[Proposition 7.2]{K1976}, a method to determine the (inner) modality of a germ with a non-degenerate Newton boundary has been described:

\begin{theorem}
Let $f\in\mathbb C\{x,y\}$ be a convenient series with a non-degenerate principle part. Construct two half lines in $\mathbb R^2_{+}$ originating at the point $(2,2)$ parallel to the $x$-axis and $y$-axis. Let $D$ be the polygon limited by a piece of the Newton polygon and the two half lines originating at the point $(2,2)$. The modality of $f$ is the number of points of $\mathbb N^2$ inside and on the boundary of $D$.
\end{theorem}

We summarize Kouchnirenko's approach in Algorithm \ref{alg:mod} in form of an elementary implementation.\footnote{Note that for large examples we can improve the performance using algorithms for computing lattice points of convex hulls. For this purpose, we rely on the \Singular\ library~ \texttt{gfanlib.so}. }

 \begin{algorithm}[ht]
\caption{Modality of a germ with a non-degenerate Newton boundary}%
\label{alg:mod}
\begin{spacing}{1.05}
\begin{algorithmic}[1]

\Require{A polynomial germ $f\in\m^2$ with a (convenient) non-degenerate Newton boundary.}

\Ensure{The modality of $f$.
}
\State Let $h$ be the maximal standard degree of a monomial with exponent vector on $\Gamma(f)$.
\State $m:=0$.
\For{$2\le i\le h$}
\For{$2\le j\le h$}
\If{$w\dash\deg(x^iy^j)\le d(f)$ }
\State $m:=m +1$.
\EndIf
\EndFor
\EndFor
\Return $m$
\end{algorithmic}
\end{spacing}

\end{algorithm}

Algorithm \ref{alg:mod} in combination with Algorithm \ref{alg:clas} allows us to compute the modality of any germ that can be transformed to a germ with a non-degenerate Newton boundary.  This approach is implemented in the 
command \texttt{modality} of the \textsc{Singular} library \texttt{arnold.lib}.

\begin{example}\label{ex mod}
Applying the \textsc{Singular} function to the polynomial $F$ of Example \ref{ex nb} we obtain:

\begin{small}
\begin{verbatim}
> modality(F);
6
\end{verbatim}
\end{small}
\end{example}

In order to compute a normal form for a given germ with non-degenerate Newton boundary via Theorem \ref{theorem:normalform}, we have to determine a regular basis for the local algebra of $f_0$ as defined in the theorem. In Algorithm~\ref{alg:regularbasis} we describe a method to do this (based on the ideal membership problem, which can be decided via Gr\"obner bases). Correctness and termination is obvious.

The computation of a regular basis of a germ with a non-degenerate Newton boundary is implemented in the function \texttt{regularBasis} of the library \texttt{arnold.lib}. 

\begin{example}
Continuing Example \ref{ex nb}, we determine a regular basis for $F$ by computing a regular basis for $f_0$. Note that this is a particular choice of a regular basis of $F$. We also determine a the basis elements on and above the Newton boundary. Note that the size of the latter, by Theorem~\ref{theorem:normalform}, agrees with the modality as computed in Example \ref{ex mod}.

\begin{small}
\begin{verbatim}
> vertexRegularBasis(F);
[ x3y5, x3y4, x4y3, x2y4, y10, x3y3, x12, y9, x11, x2y3, x3y2, y8, x10, y7, x9, x2y2, x8, 
xy3, y6, x3y, x7, xy2, y5,  x2y, x6, y4, x5, xy, x4, y3, x3, y2, x2, y, x]
\end{verbatim}
\vspace{0.1mm}
\begin{verbatim}
> vertexRegularBasisOnAndAboveNewtonPolygon(F);
[ x3y5, x3y4, x4y3, x2y4, y10, x3y3 ]
\end{verbatim}
\end{small}
\end{example}

 \begin{algorithm}[ht]
\caption{Monomial regular basis of the local algebra of a germ with a non-degenerate Newton boundary}%
\label{alg:regularbasis}
\begin{spacing}{1.05}
\begin{algorithmic}[1]

\Require{A polynomial germ $f\in\m^2$ with a (convenient) non-degenerate Newton boundary, and an integer $b$.}

\Ensure{A monomial regular basis for the local algebra of $f$ in $w(f)$-degree $\geq b$. Note that, for $b=0$ the algorithm computes a regular basis, and for $b=d(f)$ a regular basis above and on the Newton boundary.
}
\State Let $h$ be the highest $w$-weight of a monomial lying on a determinacy bound for $f$, where $w$ is the piecewise weight associated to $\Gamma(f)$.
\State$J:=\Jac(f)$
\State$B:=\emptyset$
\For{$d$ \textbf{from} $h$ \textbf{to} $b$ \textbf{by} $-1$}
\State  Let $L$ be the set of all monomials of piecewise degree $d$.
\For{$x^i y^j$ \textbf{in} $L$}
\If{$x^i y^j\not\in J$}
\State $J:=J+\left\langle x^i y^j\right\rangle$
\State $B:=B\cup\{x^i y^j\}$
\EndIf
\EndFor
\EndFor
\Return $B$
\end{algorithmic}
\end{spacing}

\end{algorithm}

An algorithm to determine a normal form which in addition explicitly finds the exceptional hypersurface in the parameter space will be stated in the next section. We also discuss how to find the admissible values of the moduli parameters of the normal form family.

\section{Computing a Normal Form  for a germ with a non-degenerate Newton Boundary.}
\label{section:normalform}
We first describe an algorithm to compute a normal form for a given input polynomial that can be transformed to a germ with a non-degenerate Newton boundary: We use Algorithm~\ref{alg:clas} to determine a polynomial with non-degenerate Newton boundary in the same right-equivalence class, we use Algorithm  \ref{alg:regularbasis} to find a regular basis for $f_0$ on and above the Newton polygon, and then apply Theorem~\ref{theorem:normalform}. 
\begin{remark}\label{remark:homogeneous}
If $f\in\mathbb C[x,y]$ be a convenient weighted homogeneous polynomial with $w=w(f)=(a,b)$, then $f\in\mathbb C[x^b,y^a]$. Furthermore, defining $\phi:\mathbb C[x^b,y^a]\to\mathbb C[t_1,t_2]$, by $x^b\mapsto t_1$ and $y^a\mapsto t_2$, the polynomial $\phi(f)$ is standard homogeneous.
\end{remark}

\begin{proof}
Since $f$ is convenient, it contains monomials $x^{b'}$ and $y^{a'}$. Since $w(f)$ is in lowest terms, $b\cdot\gcd(a,b)=b'$ and $a\cdot\gcd(a,b)=a'$. Suppose now $f$ has a mixed monomial $x^cy^d$, $c,d>1$. Then, since $f$ is weighted homogeneous of weight $ab'=a'b=\gcd(a',b')ab$, it follows that $ca+db=\gcd(a',b')ab$. Therefore $b|c$ and $a|d$, implying that $f\in\mathbb C[x^b,y^a]$. Furthermore, since $w(x^a)=w(y^b)$, it follows that $\phi(f)$ is standard homogeneous.
\end{proof}

Algorithm \ref{alg:normalform} summarizes this approach. In the algorithm we also effectively determine the valid parameter values in the  family.
 
 \begin{proof}
 Termination of the algorithm is clear. Correctness follows directly from the proof of Theorem~\ref{theorem:normalform}  and Remark \ref{remark:homogeneous}. Note that the discriminant can be computed in terms of a resultant.
 \end{proof}
 
 \begin{example}
 By applying Algorithm \ref{alg:normalform} to the polynomial $f$ of Example \ref{ex nb}, we determine a normal form for the $\mu$-constant stratum of $f$.\footnote{All relevant data computed in the course of the algorithm is collected in the \textsc{Singular} data type \texttt{NormalForm} and can be accessed through the respective keys. For more detail, please refer to the documentation of the library.}
 \begin{verbatim}
> NormalForm G = normalForm(f);
> G;
corank = 2
Milnor number = 36
modality = 6
Newton polygon:
vertices: [0,10], [2,4], [4,2], [13,0]
facets:   [[0,10], [2,4]], [[2,4], [4,2]], [[4,2], [13,0]]
normals:  [3,1], [1,1], [2,9]
parameter monomials = [ x3y5, x3y4, x4y3, x2y4, y10, x3y3 ]
normal form =  x4y2 + x2y4 + y10 + x13 
+ a(3,3)x3y3 +  a(2,4)x2y4 + a(4,3)x4y3 + a(3,4)x3y4 + a(3,5)x3y5  + a(0,10)y10 
exceptional hypersurface = (1+a(2,4))*(1+a(0,10))
\end{verbatim}
 \end{example}
 
 \begin{remark}
 Note that, after finding the normal form, the values of the moduli parameters can, in theory, be computed via an ansatz for the right equivalence (making use of finite determinacy). This is, however, not practicable except for very small examples. An efficient algorithm for this problem will be discussed in a follow-up paper.
\end{remark}

 \begin{algorithm}[h]
\caption{Normal form of a germ which is equivalent to a germ with a non-degenerate Newton boundary}%
\label{alg:normalform}
\begin{spacing}{1.05}
\begin{algorithmic}[1]

\Require{A polynomial germ $f\in\mathbb Q[x,y]$, $f\in\m^3$}  of corank $2$ with finite Milnor number.

\Ensure{A normal form for $f$, if $f$ is equivalent to a germ with non-degenerate Newton boundary, and \texttt{false} otherwise. We return a polynomial $F\in \mathbb{Q}[\alpha_1,\ldots,\alpha_m][x,y]$ representing the normal form, and a principal ideal $J\subset\mathbb C[\alpha_1,\ldots,\alpha_m]$ such that any germ $F(\alpha)(x,y)$ with $\alpha \not\in V(J)$ is in the $\mu$-constant stratum of $f$, and for any germ $g\in \mathbb{C}[x,y]$ in the $\mu$-constant stratum of $f$, there is an $\alpha \not\in V(J)$ with $g$ right equivalent to $F(\alpha)(x,y)$, and  there exist only finitely many such $\alpha$.
}
\State Replace $f$ by the output of Algorithm \ref{alg:clas} applied to $f$. \label{line:3}
\State Let $f_0$ be the sum of monomials corresponding to the vertices of $\Gamma(f)$.
\State Let $B$ be a regular basis for $f_0$  in $w(f)$-degree $\geq d(f)$, as given by Algorithm \ref{alg:regularbasis}.
\State $F:=f_0+\sum_{i=1}^{m}\alpha_i \cdot b_i$
\State $J:=\left\langle 1\right\rangle$
\For{$1\le i\le m$}
\If{$\alpha_i$ is the coefficient of a monomial of $F$ corresponding to a vertex of $\Gamma(f)$}
\State $J=J\cap \left\langle\alpha_i\right\rangle$.
\EndIf
\EndFor
\State Let $\Delta_1,\ldots,\Delta_n$ be the facets of $\Gamma(f)$.
\For{$1\le j\le n$}
\State Let $(a,b)$ be the normal vector of $\Delta_j$ in lowest terms.
\State Define $\Phi:\mathbb{Q}[x^b,y^a]\rightarrow \mathbb{Q}[t], \hspace{2mm}x^b\mapsto 1, y^a\mapsto t$.
\State $g:=\Phi(\sat(\jet(F,\Delta_j)))$
\State $J:=J\cap \left\langle\operatorname{disc}_{t}g\right\rangle$, where $\operatorname{disc}_t$ denotes the discriminant with respect to the variable $t$.
\EndFor
\Return $F,J$
\end{algorithmic}
\end{spacing}

\end{algorithm}

\end{document}